\documentclass[10pt]{article}
\usepackage{amsmath,amsthm,amsfonts,amssymb,epic,epsfig,graphicx}
\usepackage[dvips]{geometry}
\usepackage{verbatim}
\usepackage{color}

\usepackage{curves}
\usepackage{graphics}
\usepackage{verbatim}
\usepackage{appendix}

\usepackage{setspace}

\DeclareMathOperator{\cotanh}{cotanh}
\DeclareMathOperator{\sech}{sech}

\DeclareMathOperator{\arccosh}{arccosh}
\DeclareMathOperator{\arcsinh}{arcsinh}
\DeclareMathOperator{\arctanh}{arctanh}
\DeclareMathOperator{\arccoth}{arccoth}

\newcommand{\reals}{\mathbb{R}}

\theoremstyle{plain}

\newtheorem{stel}{Theorem}
\newtheorem{gevolg}{Corollary}
\newtheorem{lemma}{Lemma}
\theoremstyle{remark}
\newtheorem{opm}{Remark}

\author{Patrick De Leenheer\footnote{Department of Mathematics, University of Florida, email: deleenhe@ufl.edu}}

\begin{document}

\title{Optimal placement of Marine Protected Areas}

\date{}

\maketitle

\begin{abstract}
Overfishing can lead to the reduction or elimination of fish populations and the degradation or even destruction of their habitats. 
This can be prevented by introducing Marine Protected Areas (MPA's), regions in the ocean or along coastlines where fishing is controlled. MPA's can also lead to larger fish densities outside the protected area through spill-over, 
which in turn may increase the fishing yield. A natural question in this context, is where exactly to establish 
an MPA, in order to maximize these benefits. This problem is addressed along a one-dimensional stretch of coast-line, 
by first proposing a model for the fish dynamics. Fish are assumed to move diffusively, and are subject to recruitment, natural death and harvesting through fishing. The problem is then cast as an optimal control problem 
for the steady state equation corresponding to the PDE which models the fish dynamics. The functional being maximized 
is a weighted sum of the average fish density and the average fishing yield. It is shown that optimal controls exist, and that 
the form of an optimal control -and hence the location of the MPA- is determined by two key model parameters, namely 
the size of the coast, and the weight of the average fish density appearing in the functional. If these parameters are large enough 
-and precisely how large, can be calculated exactly- the results indicate when and where an MPA should be established. The results indicate that an MPA always takes the form of a Marine Reserve, where fishing is prohibited. 
The main mathematical tool used to prove the results is Pontryagin's maximum principle.

\end{abstract}

\section{Introduction}
Marine Protected Areas (MPA's), see \cite{gerber2003,white11} for reviews, are regions in oceans or along coastlines where fishing is controlled, whereas Marine Reserves are regions were fishing is prohibited.
MPA's have been proposed 
as a fisheries management tool and contrast the more traditional approaches which rely on limiting spatially uniform 
harvesting rates. 
MPA's are advocated to prevent extinction of endangered or over-fished species, to restore their habitats (for instance, by preventing 
fishing nets to be  dragged along the ocean floor), and to promote species diversity. On the other hand, MPA's 
could damage the interests of commercial fisheries by decreasing overall fishing yields.
Although 
these arguments appear to be perfectly reasonable intuitively, there are several other factors that could influence the impact of MPA's:
\begin{itemize}
\item Although earlier mathematical models \cite{mangel98,hastings99} did not include spatial structure explicitly, more recent 
modeling incorporates spatial effects for both larvae \cite{botsford01,gaines2003} and adults \cite{moffitt09,langebrake}.
Some species such as crustaceans or reef fish, are fairly sedentary, whereas others are highly mobile. If MPA's increase 
density or biomass of fish inside the protected area, then 
more mobile species can easily escape the protection afforded by an MPA. Therefore, they are more easily exposed 
to being caught, which may lead to increasing, rather than decreasing fishing yields. In addition, the non-homogeneous nature of fish habitats, also impacts the design of MPA networks \cite{crowder2000}.
\item
More often than not, fish are part of intricate food webs. It is reasonable to expect that the effect of an MPA on a specific species 
depends on the role it plays in the web (predator, competitor, mutualist or resource) \cite{baskett07,kellner2011}.
\item
Many species are also characterized by different stages, such as larval, juvenile and mature or reproductive stages, each with their own movement patterns. The impact of an MPA will most likely depend on the interplay between the these spatially and temporally varying scales.
\cite{botsford01,gaines2003,mangel2005}.
\item
On relatively short time scales, fisheries can create evolutionary pressure to select for species reaching maturity sooner, but leading to smaller-sized adults and therefore to lower biomass, thus negatively affecting yields. MPA's are believed to mitigate these adverse effects, promoting 
sustainable fisheries \cite{baskett05,dunlop2009}.
\end{itemize}
These additional factors suggest that the decision to implement an MPA is multi-faceted, and needs to be based on appropriate 
measures which can be investigated by both mathematical models, and analysis of field data.

The purpose of this paper is to present a mathematical framework to aid 
in the decision of whether or not it would be beneficial to introduce an MPA, and if so, where to implement it. A novel objective measure capturing 
the effect of the MPA is proposed. It takes the form of a weighted sum that consists of the yield on one hand, and the average fish density on the other. 
The motivation for choosing this measure is that it incorporates two of the main measures that have been used in the past, 
namely yield and density, coupled with the fact mentioned earlier, that MPA's are believed to have opposite effects on these measures. 
This leads to a trade-off problem, and the natural context to consider it is provided by optimal control theory \cite{lee-markus,sontag,liberzon}. In fact, optimal control theory 
offers a much wider array of possible harvesting strategies, than the ones provided by Marine Reserves where  fishing is prohibited in specific locations, or by the classical approach where fishing is allowed everywhere, but at a lower-than-maximal  capacity. Indeed, in an optimal control framework, it is a priori allowed that at every location along the coastline, the harvest rate takes on an arbitrary 
value between a prescribed minimum and maximum value, effectively blending the two approaches. This implies that the optimal strategy could be a complicated function of the location along the coast line, at least in principle. 
However, the analysis shows that there are scenarios where the optimal strategy is to allow fishing at maximal capacity everywhere. 
This happens when the weight of the average density in the measure being maximized is below a certain threshold. On the other hand, 
when the weight is above the threshold, the optimal strategy is still to allow fishing everywhere at maximal capacity, if the size of the 
coastline is small. But, there is a critical value for the coastline size, above which the optimal strategy consists of installing a single Marine Reserve in the middle of the coast line. In other words, although a priori, there is no reason why the optimal strategy should take the form of a Marine Reserve, the analysis reveals that the best strategy corresponds to the implementation of a single (as opposed to a network of)  Marine Reserves. This is a feature common to many optimal control strategies, known as the bang-bang principle \cite{lee-markus,sontag,liberzon}. 

This paper is not the first to propose the use of optimal control in the context of MPA's, and follows the lead of \cite{neubert}, perhaps the first paper to advocate this approach. However, both the model, the analysis, and the results here, deviate from those in \cite{neubert} in several respects. First, in \cite{neubert} only 
the yield is maximized. Here, as mentioned earlier, a linear combination of yield and density is maximized, effectively incorporating the influence exerted by conservationist groups. 
Secondly, both papers propose a reaction-diffusion equation to model the fish dynamics, and although they share the same diffusion term which describes how fish move, they use different reaction terms. The one proposed here is an affine function of the fish density, as 
opposed to the logistic one used in \cite{neubert}. The latter is appropriate to model closed systems, which are subject to density dependent reproduction, whereas the former is more suited to the modeling of open systems, where recruitment is spatially uniform through the  well-mixing of 
larvae, perhaps caused by ocean circulation. Nevertheless, both models assume the same harvesting term, which is part of the reaction term in the PDE. Thirdly, the logistic term in \cite{neubert} leads to nonlinear systems in the optimal control problem, that cannot be solved explicitly, and therefore \cite{neubert} resorts to 
an analysis of their numerical solution. Here on the other hand, the optimal control problem can be solved analytically, and therefore 
the structure of the optimal control can be calculated and analyzed in terms of the model parameters. 
Finally, the conclusions of both papers are qualitatively different. As pointed out in \cite{neubert}, optimal harvesting strategies may consist of complicated networks of MPA's -in some cases there are infinitely many disjoint regions along the finite 
coastline where fishing is prohibited- a scenario which cannot be implemented in practice. Here on the other hand, the optimal strategy is always very simple, and consists of at most one MPA located in the middle of the coastal region. Moreover, thanks to the availability 
of certain implicit formulas, the location of the MPA can be determined, which 
is only possible numerically in \cite{neubert} for reasons mentioned earlier.

\section{Model and problem formulation}
Consider the following model:
\begin{eqnarray}
U_T&=&DU_{XX}+R-\mu U -H(X)U,\;\; -L/2<X<L/2,\nonumber\\ 
U(-L/2,T)&=&U(L/2,T)=0 \textrm{ for all }T\geq 0\label{unscaled}
\end{eqnarray}
Here, points along the scalar coastline of length $L>0$ are represented by the spatial variable $X$ taking values in 
the interval $[-L/2,L/2]$, and $U(X,T)$ denotes the fish density 
at location $X$ and time $T$. The boundary condition corresponds to a lethal or absorbing boundary, where fish cannot survive. 
Although it is not considered here, a no-flux, or Neumann boundary condition could be assumed instead,  corresponding to a scenario where fish cannot enter or exit via the boundary. Similarly, a mixed-type boundary 
condition could be investigated as well, when recruitment of adult fish also occurs through the boundary. 
The fish are assumed to diffuse with a constant diffusion constant $D>0$, are recruited at constant 
rate $R>0$, die at constant per capita rate $\mu>0$ and are harvested at per capita rate $H(X)$ which depends on 
the location $X$. We note that this model does not include any density-dependent features, since recruitment 
occurs at a constant rate $R$ in space and time, and is independent of the current fish density $U$. This scenario is motivated by reef fish whose habitats are restricted to specific reef patches. The boundary of such a patch is lethal,  possibly due to the presence of a predator patrolling the patch boundary. Recruitment happens after larvae have settled in the patch. The assumption of a uniform recruitment rate corresponds to a case where adult fish abundantly generate larvae over many reef patches, which in turn are dispersed over these patches by diffusion and/or advection due to ocean currents. For a similar model which takes a specific density-dependent functional form into account, see 
\cite{neubert} where a logistic growth form is assumed (that is, the reaction term $R-\mu U$ is replaced 
by a logistic $rU(1-U/K)$ for positive constants $r$ and $K$). 

Since the fishermen's fleet is limited, it is reasonable to assume that $H$ takes values in the interval $[0, {\bar H}]$, where ${\bar H}>0$ denotes the 
maximal harvesting rate. In  what follows the following notation for the average of a function 
$F(X)$, defined on the interval $[-L/2,L/2]$, is used:
$$
<F>\;\; := \; \frac{1}{L}\int_{-L/2}^{L/2}F(X)dX
$$

The problem  addressed here is to find the function $H(X)$ which maximizes the steady state functional:
\begin{equation}\label{unscaled-cost}
J(H(X))=<H(X)U(X)>+Q<U(X)>,
\end{equation}
where $Q\geq 0$ is a fixed weight parameter, and $U(X)$ is a steady state of $(\ref{unscaled})$ using $H(X)$.
This functional reflects the tradeoff between two terms. The first one, $<HU>$, 
represents the average harvest by fishermen or average yield, whereas the second denotes the average fish density, multiplied by a parameter $Q$ which measures the relative weight attributed to the average fish density compared to the average yield. It will be small in regions whenever there is little pressure by conservationists to limit fishing, and large otherwise. It is determined by  
local customs, traditions, culture and politics, and agencies that regulate fishing. Note that $Q$ has the dimension 
of a rate, since the first term, the fishing yield, is a rate as well (average amount of fish caught per unit of time). 
It is the purpose of this paper to 
assist  in making recommendations on placing MPA's in a way that maximizes 
$J$. 
Note also that no costs related to fishing activities are incorporated in the functional $J$. For more traditional  bio-economic theories focused on 
trade-offs between fishing yields and fishing costs, but neglecting the effect of conservationist pressure, see 
\cite{clark,alain}.

Several model parameters can be scaled out as follows. Defining
\begin{eqnarray}
u&=&U/(R\mu^{-1}),\; t=\mu T,\; x=X/\sqrt{D\mu^{-1}},\; l=L/(\sqrt{D\mu^{-1}}), \nonumber\\ 
h(x)&=&H(X)/\mu,\; {\bar h}={\bar H}/\mu,\; q=Q/\mu,\; j=J/R \label{scale}
\end{eqnarray}
yields the scaled model:
\begin{eqnarray}
u_t&=&u_{xx}-\left(1+h(x)\right)u +1,\;\;  -l/2<x<l/2,\nonumber \\ 
u(-l/2,t)&=&u(l/2,t)=0,\textrm{ for all } t\geq 0\label{scaled}
\end{eqnarray}
with scaled functional:
\begin{equation}\label{scaled-cost}
j(h(x))=<h(x)u(x)>+q<u(x)>,
\end{equation}
which needs to be maximized over functions $h(x)$ taking values in $[0,{\bar h}]$, and where $u(x)$ is the steady 
state of $(\ref{scaled})$ corresponding to $h(x)$. Note that now the averages appearing in the scaled functional, are averages over the scaled interval $[-l/2,l/2]$.
The scaled problem contains only $3$ parameters: 
the scaled weight parameter $q\geq 0$, 
the coastal length $l>0$,  and the scaled maximum harvesting rate ${\bar h}>0$. The main results will first be phrased 
in terms of these parameters, but later in the Discussion Section they are translated into statements in terms of the parameters of the unscaled problem, using the scaling laws in $(\ref{scale})$. 
Also  note that the scaled functional could be rewritten more compactly as $j({\tilde h}(x))=<{\tilde h}(x)u(x)>$, where ${\tilde h}(x)$ takes values in the interval $[q,q+{\bar h}]$, although this point of view is not considered in what follows.

\section{Existence of optimal controls}
The steady state problem associated to $(\ref{scaled})$ 
can be re-written as a system of two first-order equations by letting $v=u'$ where $'$ denotes $d/dx$:
\begin{eqnarray}
u'&=&v \label{ss1}\\
v'&=&\left(1+h(x)\right)u -1 \label{ss2}
\end{eqnarray}
with boundary condition
\begin{equation}\label{BC-state}
u(-l/2)=u(l/2)=0
\end{equation}
The problem is to find a measurable function $h(x)$, taking values in the interval $[0,{\bar h}]$ for a.e. x in $[-l/2,l/2]$ such that for this particular choice of $h(x)$, a solution $(u(x),v(x))$ to $(\ref{ss1})-(\ref{ss2})$ exists that satisfies the 
boundary condition $(\ref{BC-state})$, and the constraint that $u(x)\geq 0$ for all $x\in[-l/2,l/2]$. The latter constraint is natural since $u$, the scaled fish density, cannot be negative. Every function $h(x)$ satisfying these requirements is called an {\it admissible control} (note that, at least in principle, such functions might not even exist), and the problem is now to find an admissible control which maximizes the scaled functional $(\ref{scaled-cost})$. Existence results for optimal controls are discussed in \cite{lee-markus,berkovitz-paper,seierstad}.

\begin{stel}\label{existence}
There exists an admissible control $h^*(x)$ defined for $x\in [-l/2,l/2]$, which maximizes the scaled functional $(\ref{scaled-cost})$.
\end{stel}
\begin{proof}
The proof follows from a Corollary to Theorem $4$ in \cite{lee-markus}. It will be shown that the assumptions $(1)-(5)$ and conditions (a), (b) and (c) from that Theorem hold: 
(1) {\it The initial set $X_{-l/2}$ and target set $X_{l/2}$ are compact.}  
In the problem considered here, the initial and target sets follow from the boundary condition $(\ref{BC-state})$, namely $X_{-l/2}=\{(u,v)|u=0\}=X_{l/2}$. Clearly, these sets are not compact, but it will be shown that without loss of generality they can be redefined to be compact sets as follows:
\begin{equation}\label{initial-target}
X_{-l/2}=\{(u,v)|u=0,0\leq v\leq 1\}\textrm{ and } X_{l/2}=\{(u,v)|u=0,-1\leq v \leq 0\}.
\end{equation}
To see this, note that every solution of $(\ref{ss1})-(\ref{ss2})$ initiated in 
the domain
$$
R=\{(u,v) | u\geq 0, v\geq 0, u+v\geq 1\},
$$
remains in $R$ for every measurable $h$ which takes values in $[0,{\bar h}]$ for a.e. $x$ in $[-l/2,l/2]$. Moreover, 
the vector field of system $(\ref{ss1})-(\ref{ss2})$ is transversal to the boundary part of $R$ where $u=0$. This implies that no solution initiated on 
this part of the boundary of $R$ at $x=-l/2$ can reach the set $\{(u,v)|v=0\}$ at $x=l/2$. Furthermore, every solution 
initiated at $x=-l/2$ in an initial condition $(0,v)$ with $v<0$ is such that $u$ becomes negative for $x$ larger then, and near $-l/2$, violating the state constraint. These facts show that there is no loss of generality in choosing the initial set $X_{-l/2}$ as in $(\ref{initial-target})$. 
A similar argument shows that the target set can be chosen to be 
$X_{l/2}=\{(u,v)|u=0,-1\leq v \leq 0\}$ (but now the reflection of the domain $R$ with respect to the 
$u$-axis should be considered, which turns out to be backward invariant).
(2) {\it The control set $[0,{\bar h}]$ is nonempty and compact}.  This is obvious.
(3) {\it The set of state constraints is described by a family of inequality constraints $g_1(x)\geq 0,\dots,g_r(x)\geq 0$, where each function $g_i$ is continuous}. In the problem considered here, there is a single state constraint given by $g(u,v)\geq 0$, where $g(u,v):=u$, which is continuous.
(4) {\it Let ${\cal F}$ denote the family of admissible controls}. This is merely the introduction of the notation for the family of admissible controls.
(5) {\it For each admissible control $h(x)$, the functional is given by: 
$$
j(h(x))=\int_{-l/2}^{l/2}\frac{(q+h(x))u(x)}{l}dx,
$$
where the above integrand $(1+h)u/l$ is a $C^1$ function in the variables $(u,v,h)$.}  This is obvious as well.

The conditions (a), (b) and (c) of Theorem $4$ in \cite{lee-markus} are next: (a) {\it The family ${\cal F}$ is not empty}. 
It is claimed that all the constant functions $h(x)={\hat h}$ for all $x$ in $[-l/2,l/2]$, where ${\hat h}$ is an arbitrarily chosen number in the interval $[0,{\bar h}]$, belong to ${\cal F}$. To prove this it suffices to show that for every ${\hat h}$, there exists a solution $(u_{{\hat h}}(x),v_{{\hat h}}(x))$, $-l/2\leq x\leq l/2$, to system $(\ref{ss1})-(\ref{ss2})$ with $h(x)={\hat h}$, such that $u_{{\hat h}}(x)\geq 0$ for $x\in [-l/2,l/2]$, and $(\ref{BC-state})$ holds. It is easily verified that the following pair
$$
\begin{pmatrix}
u_{{\hat h}}(x)\\
v_{{\hat h}}(x)
\end{pmatrix}
=
\begin{pmatrix}
\frac{1}{1+{\hat h}}\left[1-\sech \left(\sqrt{1+{\hat h}}l/2\right)\cosh \left(\sqrt{1+{\hat h}}x\right) \right]\\
-\frac{\sech \left(\sqrt{1+{\hat h}}l/2\right)}{\sqrt{1+{\hat h}}}\sinh \left(\sqrt{1+{\hat h}}x\right)
\end{pmatrix}
$$
satisfies these constraints. (b) {\it There exists a uniform bound $b$ for every solution $(u(x),v(x))$, $-l/2\leq x\leq l/2$, in response to an arbitrary control $h(x)$ in ${\cal F}$:}
$$
||(u(x),v(x))||\leq b,\textrm{ for } -l/2\leq x \leq l/2
$$
This follows immediately from the variation of parameters formula for the solution of $(\ref{ss1})-(\ref{ss2})$, given that the initial set $X_{-l/2}$ is compact, that $h$ takes values in the compact set $[0,{\bar h}]$ a.e., and because $x\in [-l/2,l/2]$, a bounded set. 
(c) {\it The extended velocity set
$$
V(u,v):=\left \{ \left(-\frac{(q+h)u}{l},v,(1+h)u-1\right) \vert h \in [0,{\bar h}] \right \}
$$
is convex in $\reals^3$ for each fixed $(u,v)$.} This follows from the fact that for each fixed $(u,v)$, 
the image of the convex set $[0,{\bar h}]$ under the affine 
map $h\rightarrow \left(-\frac{(q+h)u}{l},v,(1+h)u-1\right)$, is convex.

\end{proof}

\section{The maximum principle and some preliminary properties of an optimal control}

Introduce the Hamiltonian associated to $(\ref{ss1})-(\ref{ss2})$ and functional $(\ref{scaled-cost})$:
\begin{equation}\label{Hamiltonian}
H(u,v,\lambda_1,\lambda_2,h)=\frac{1}{l}(h+q)u+\lambda_1v+\lambda_2 \left(\left(1+h\right)u -1\right),
\end{equation}
where $(\lambda_1,\lambda_2)$ are called the adjoint variables. These play a role which is similar to Lagrange multipliers, which appear 
in the search for extrema of functions under constraints. 

By Pontryagin's maximum principle \cite{lee-markus}, any maximum for the functional attained at some
\newline 
$\left(u^*(x),v^*(x),\lambda_1^*(x),\lambda_2^*(x),h^*(x)\right)$ must maximize the Hamiltonian with respect to $h(x)$:
\begin{equation}\label{max-Ham}
H(u^*(x),v^*(x),\lambda_1^*(x),\lambda_2^*(x),h(x))\leq H(u^*(x),v^*(x),\lambda_1^*(x),\lambda_2^*(x),h^*(x)), 
\end{equation}
for all $x$ in $[-l/2,l/2]$ and $h(x)$ in $[0,{\bar h}]$, and must solve the Hamiltonian system
\begin{eqnarray}
u'&=&\partial H/\partial \lambda_1=v \label{Ham-sys1}\\
v'&=&\partial H/\partial \lambda_2=(1+h(x))u-1\label{Ham-sys2}\\
\lambda_1'&=&-\partial H/\partial u=-(h(x)+1)\lambda_2-\frac{h(x)+q}{l}\label{Ham-sys3}\\
\lambda_2'&=&-\partial H/\partial v=-\lambda_1\label{Ham-sys4}
\end{eqnarray}
with boundary and transversality conditions:
\begin{eqnarray}
u(-l/2)&=&u(l/2)=0\label{BC}\\
\lambda_2(-l/2)&=&\lambda_2(l/2)=0\label{TC}
\end{eqnarray}

\begin{opm} {\bf (Lack of abnormal multiplier)}

In principle,  the following Hamiltonian should have been considered:
\begin{equation}\label{Hamiltonian2}
{\hat H}(u,v,\lambda_0,\lambda_1,\lambda_2,h)=\lambda_0\frac{1}{l}(h+q)u+\lambda_1v+\lambda_2 \left(\left(1+h\right)u -1\right),
\end{equation}
where $\lambda_0\geq 0$, is a constant, known as the {\it abnormal multiplier}. In this case, the the adjoint equation reads:
\begin{eqnarray}
\lambda_1'&=&-\partial {\hat H}/\partial u=-(h(x)+1)\lambda_2-\lambda_0\frac{h(x)+q}{l}\label{Ham-sys-new1}\\
\lambda_2'&=&-\partial {\hat H}/\partial v=-\lambda_1\label{Ham-sys-new2}
\end{eqnarray}
with transversality conditions $(\ref{TC})$, and the maximum principle implies the existence of a nonzero triple \cite{liberzon}:
\begin{equation}\label{nontrivial}
\begin{pmatrix}
\lambda_0\\
\lambda_1(x)\\
\lambda_2(x)
\end{pmatrix}\neq \begin{pmatrix}
0\\
0\\
0
\end{pmatrix}, -l/2 \leq x \leq l/2,
\end{equation}
where $(\lambda_1(x),\lambda_2(x))$ solves $(\ref{Ham-sys-new1})-(\ref{Ham-sys-new2})$ with transversality conditions $(\ref{TC})$.  
It can be shown by contradiction that necessarily $\lambda_0\neq 0$. The proof is straightforward and omitted.
Since $\lambda_0\neq 0$, it can be rescaled to take the value $1$. This explains why the Hamiltonian $(\ref{Hamiltonian2})$ was specialized to $(\ref{Hamiltonian})$.
\end{opm}
\begin{opm} \label{pos-constraint} {\bf (State constraint)} 

Another simplification is that in the definition of the Hamiltonian, the natural non-negativity constraint $u\geq 0$ which a solution corresponding to an optimal control should posses, was not included. For more on maximum principles for systems with state constraints, see 
\cite{knowles,seierstad,evans}. 
However, it can be shown that this constraint is automatically satisfied. In fact, 
if $u^*(x)$ is the $u$-component of a solution of the Hamiltonian system $(\ref{Ham-sys1})-(\ref{Ham-sys4})$, 
then 
\begin{equation}\label{u-pos}
u^*(x)>0,\textrm{ for all } x\in (-l/2,l/2). 
\end{equation}
\end{opm}

Returning to the application of the maximum principle using the Hamiltonian in $(\ref{Hamiltonian})$. 
Since $H$ is linear in the control variable $h$, it follows from $(\ref{max-Ham})$ that
\begin{equation}
h^*(x)=\begin{cases}\label{optimal}
0, \textrm{ if } u^*(x)(1/l+\lambda_2^*(x))<0\\
{\bar h}, \textrm{ if } u^*(x)(1/l+\lambda_2^*(x))>0
\end{cases}
\end{equation}
The set of points $\{(u,v,\lambda_1,\lambda_2)|u=0\textrm{ or }\lambda_2=-1/l\}$ is called the {\it switching surface} of the Hamiltonian system. 
In Remark $\ref{pos-constraint}$ it was already shown that an optimal solution $(u^*(x),v^*(x),\lambda_1^*(x),\lambda_2^*(x))$ cannot belong to the part of the switching surface where $u=0$, other than at the initial and final times $x=\pm l/2$ of the control horizon, see 
in particular $(\ref{u-pos})$. 
Notice that this fact, combined with $(\ref{optimal})$ and the transversality condition $(\ref{TC})$, also shows that $h^*(x)={\bar h}$ for all $x$ near $x=-l/2$ and $x=l/2$. 
In other words, both at the beginning and end of the control horizon, an optimal control equals its maximal value ${\bar h}$. 
In addition, $(\ref{u-pos})$ implies that the optimal control $h^*(x)$ takes the form:
\begin{equation}
h^*(x)=\begin{cases}\label{optimal2}
0, \textrm{ if } \lambda_2^*(x)<-1/l\\
{\bar h}, \textrm{ if } \lambda_2^*(x))>-1/l
\end{cases}
\end{equation}

The question remains what values $h^*(x)$ takes for other $x$-values in the control interval. In other words, will 
the state $(u^*(x),v^*(x),\lambda_1^*(x),\lambda_2^*(x))$ of the Hamiltonian system ever cross, or remain on the (smaller) switching surface
\begin{equation}\label{switch-surface}
{\cal S}=\{(u,v,\lambda_1,\lambda_2)|\lambda_2=-1/l\}
\end{equation}
If the state remains on ${\cal S}$ for $x$ in some subinterval of $[-l/2,l/2]$, then the value of $h^*(x)$ is not determined by 
$(\ref{optimal})$. The control is said to be {\it singular} if this happens, and a more detailed analysis would be required to determine $h^*(x)$. 
However, it is claimed that here, optimal controls cannot be singular. Indeed, suppose that $\lambda_2^*(x)=-1/l$ for all $x$ in some interval ${\cal I}\subset [-l/2,l/2]$. 
To keep the solution on ${\cal S}$ during ${\cal I}$, requires that $(\lambda_1^*)'=0=(1-q)/l$, which is impossible if $q\neq 1$.  
If $q=1$, the adjoint problem reduces to:
\begin{eqnarray*}
{\lambda_1}'&=&-(h(x)+1)[\lambda_2+1/l]\\
{\lambda_2}'&=&-\lambda_1
\end{eqnarray*}
To keep the solution on ${\cal S}$, requires that $\lambda_1(x)=0$ for all $x$ in ${\cal I}$. This is possible indeed, but it implies that the state of the adjoint system is at the 
steady state $(\lambda_1,\lambda_2)=(0,-1/l)$, and thus remains there for all $x\in[-l/2,l/2]$. 
This in turn implies that the transversality condition $(\ref{TC})$ cannot be satisfied. Thus, optimal controls cannot be singular.

Although some features of an optimal control have emerged from the foregoing discussion, the Hamiltonian system needs to be investigated more closely, in order to better understand the structure of an optimal control. 
This is done in the next section.

\section{Solving the Hamiltonian system}
In the previous section it was shown that optimal controls are never singular, and also that 
near $x=\pm l/2$, an optimal control takes its maximal value ${\bar h}$. The question arises whether 
an optimal control remains constant and equal to ${\bar h}$ on the entire control horizon $[-l/2,l/2]$, or whether it 
ever switches to zero at some point when the state of the Hamiltonian system crosses ${\cal S}$. 
Notice that if this happens, an optimal control must necessarily switch back to ${\bar h}$ at least once 
later during in the interval $[-l/2,l/2]$, since it must equal ${\bar h}$ for all $x$ near $-l/2$. To preview the main results, they are briefly 
summarized here.  
A bound for $q$ will be calculated, such that below this bound, there is no switch. If $q$ exceeds this bound, on the 
other hand, no switches occur if $l$ falls below some threshold, and exactly two switches 
occur when $l$ is above it. 

If $h(x)={\bar h}$ or $h=0$ for all $x$, then the adjoint system is a linear time-invariant system of the form
\begin{eqnarray}
\lambda_1'&=&-a\lambda_2 - \frac{b}{l} \label{saddle1}\\
\lambda_2'&=&-\lambda_1\label{saddle2}
\end{eqnarray}
for suitable $a>0$ and $b>0$. Straightforward calculations establish the behavior of this system:
\begin{lemma}\label{saddle}
System $(\ref{saddle1})-(\ref{saddle2})$ has a unique equilibrium point $E=(0,-b/(al))$ which is a saddle. The stable manifold 
is a line through $E$ with slope $1/\sqrt{a}$, and the unstable manifold is a line through $E$ with slope $-1/\sqrt{a}$.
\end{lemma}
The phase portrait of system $(\ref{saddle1})-(\ref{saddle2})$ is illustrated in Figure $\ref{saddle-portrait}$.
\begin{figure}[ht]
\centering
\includegraphics[width=5in]{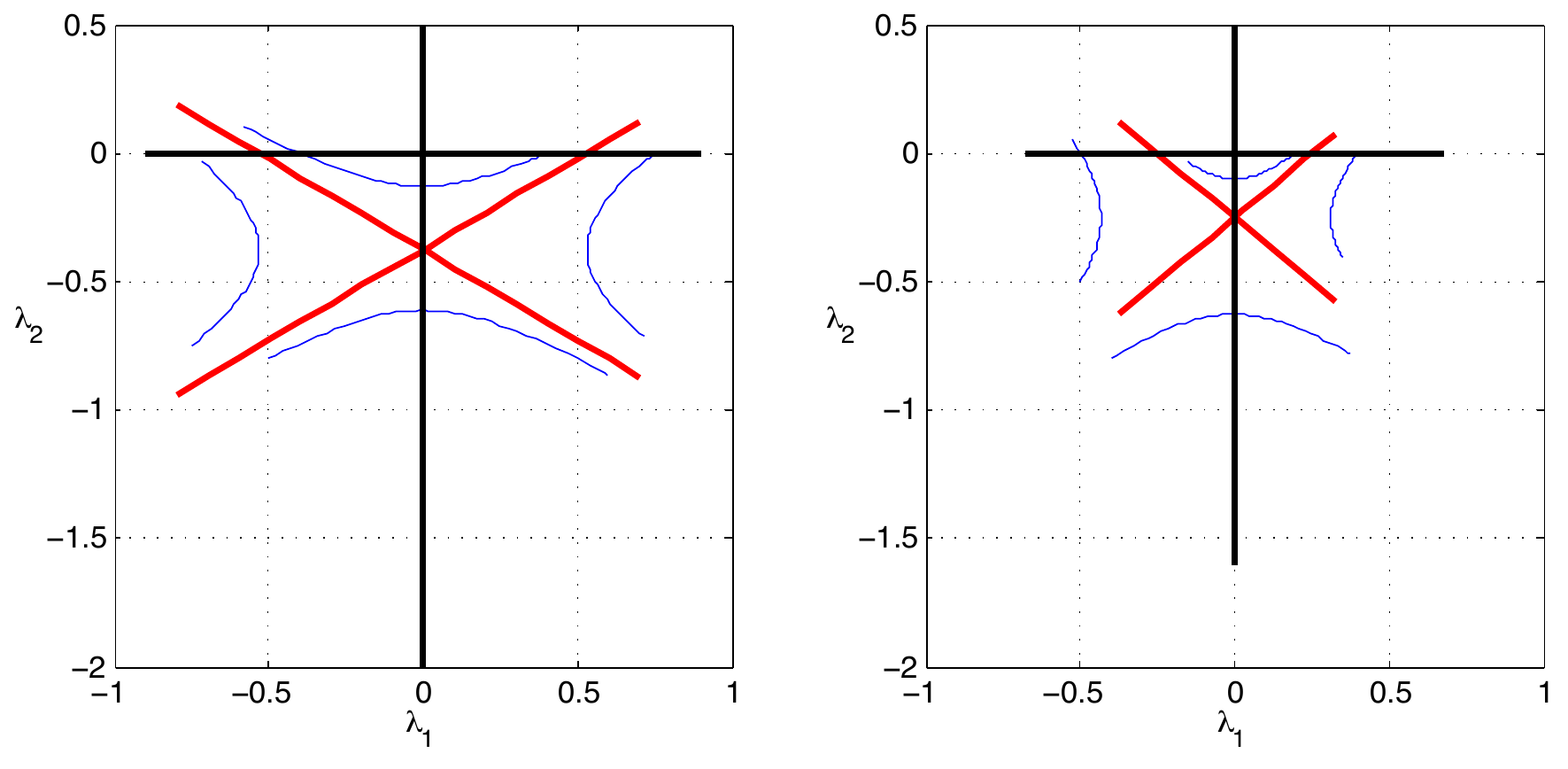}
\caption{Phase portrait of system $(\ref{saddle1})-(\ref{saddle2})$, with $a=2$, $b=1.5$ and $l=2$ on the left, and 
with $a=1$, $b=0.5$ and $l=2$ on the right. 
Stable manifolds are the straight lines with positive slope (in red), whereas unstable manifolds are the straight lines with negative slope (also in red).}
\label{saddle-portrait}
\end{figure}

\subsection{The case $0<q\leq 1$.}
It is first shown that switches in the value of $h(x)$ are not possible in this case. 

Consider the left panel of Figure $\ref{no-switch}$, which depicts some orbits of the adjoint system $(\ref{Ham-sys3})-(\ref{Ham-sys4})$ when $h(x)={\bar h}$.  Notice that the steady state has coordinates $(0,-({\bar h}+q)/({\bar h}+1)l)$, and thus it does not lie below the switching line 
$\{(\lambda_1,\lambda_2)|\lambda_2=-1/l\}$. The problem is to determine whether or not there are solutions starting on the $\lambda_1$-axis at $x=-l/2$ which reach the horizontal switching line $\{(\lambda_1,\lambda_2)|\lambda_2=-1/l\}$ at some $x_s< l/2$. As it turns out, there are no such solutions, and a proof is briefly sketched next. 
\begin{itemize}
\item
If $\lambda_1(-l/2)\leq 0$, this is impossible, as the solution will remain in the second quadrant  because it is forward invariant. Notice also 
that the transversality $(\ref{TC})$ at $x=l/2$ cannot hold for such solution.
\item
If $0<\lambda_1(-l/2)\leq \lambda_s$, where $\lambda_s:=({\bar h}+q)/(\sqrt{{\bar h}+1}l)$ is the intercept of the stable manifold (the straight line with positive slope (in red) 
in Figure $\ref{no-switch}$), this is also impossible. Indeed, this follows because the region that lies above the stable and unstable manifold 
(the straight line with negative slope (in red) in Figure $\ref{no-switch}$) is forward invariant, and because the lowest point of this region 
-the steady state- does not lie below the switching line.
\item
If $\lambda_1(-l/2)>\lambda_s$, the solution may reach the switching line at some $x_s<l/2$. Assume it happens and 
denote the state of the adjoint system at $x=x_s$ by $(\lambda_1(x_s),-1/l)$. Note that necessarily 
$\lambda_1(x_s)\geq 0$ because the region which is part of the fourth quadrant which lies below the stable manifold 
and above the unstable manifold, is forward invariant. 
When the solution reaches the switching line, it will cross it, and thus the control variable $h$ now switches from ${\bar h}$ to $0$. Thus the adjoint system becomes $(\ref{Ham-sys3})-(\ref{Ham-sys4})$ but now with $h(x)=0$, whose orbits are depicted in the right panel of Figure $\ref{no-switch}$. From the orbits it is clear that 
for all $x>x_s$, the solution will remain below the switching line $\lambda_2=-1/l$. This follows from the fact that the region $\{(\lambda_1,\lambda_2)|\lambda_1\geq 0, \lambda_2\leq -1/l\}$ is forward invariant. Indeed, if $\lambda_1=0$, then ${\dot \lambda_1}=-\lambda_2-q/l\geq (1-q)/l\geq 0$, whereas if $\lambda_2=-1/l$, then ${\dot \lambda_2}=-\lambda_1\leq 0$. Therefore, a solution 
that ever reaches the switching line, will remain below it for all future times, and 
hence it can never satisfy the transversality condition $(\ref{TC})$ at $x=l/2$. Thus, the possibility that an optimal control exhibits a switch,  has been ruled out.
\end{itemize}
\begin{figure}[ht]
\centering
\includegraphics[width=5in]{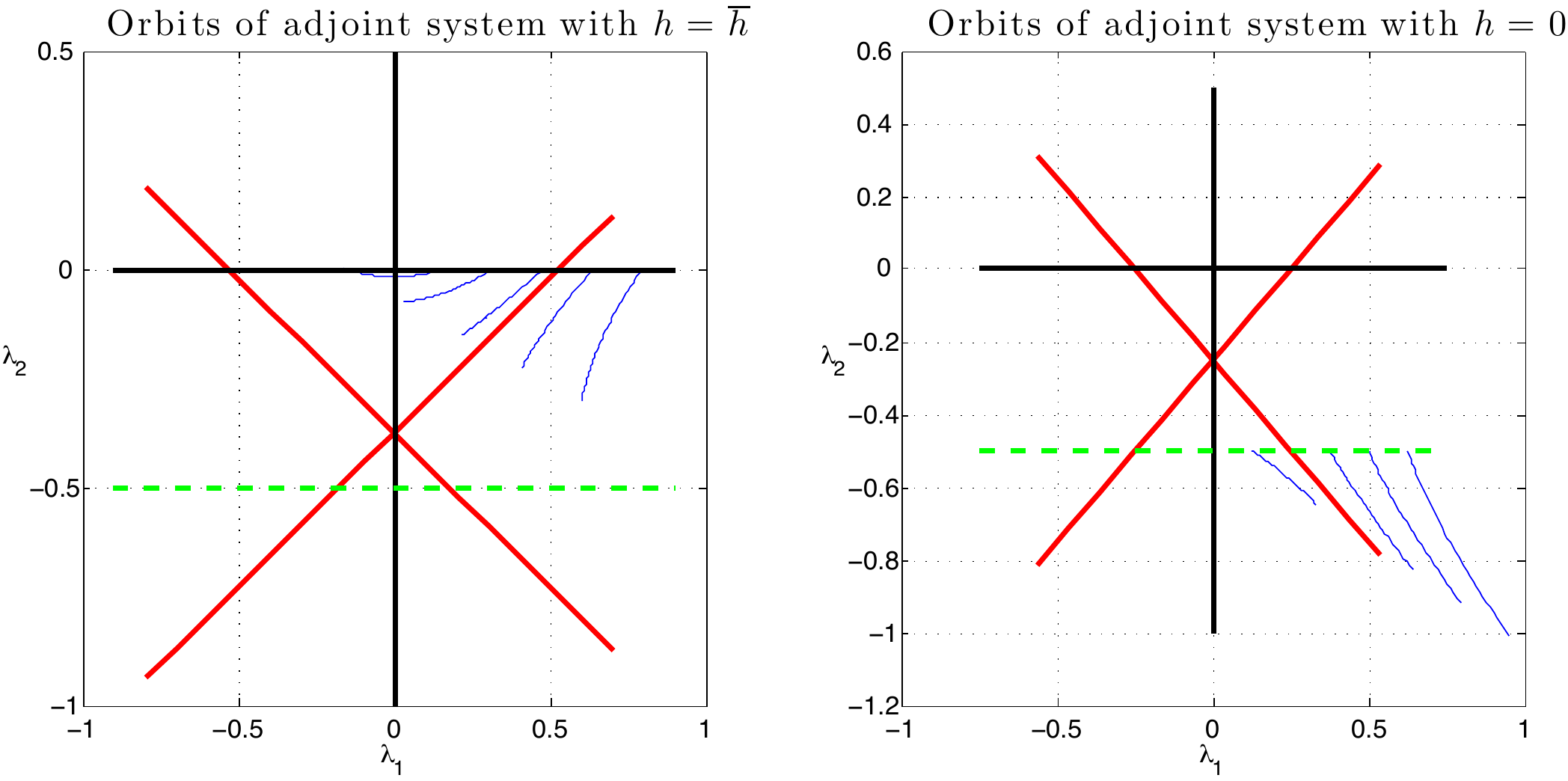}
\caption{Phase portrait of system $(\ref{Ham-sys3})-(\ref{Ham-sys4})$, with $l=2$, $q=0.5$ and $h={\bar h}=1$ on the left, and 
$h=0$ on the right. Stable manifolds are the straight lines with positive slope (in red), whereas unstable manifolds are straight lines with negative slope (also in red). 
The dashed curve (in green) is the switching line where $\lambda_2=-1/l$.}
\label{no-switch}
\end{figure}
The foregoing discussion shows that every solution of the adjoint system $(\ref{Ham-sys3})-(\ref{Ham-sys4})$ with $h(x)={\bar h}$
which satisfies the transversality conditions $(\ref{TC})$, must be such that $0<\lambda_1(-l/2)<\lambda_s$. In fact, it will be shown that such a solution always exists, and moreover that it is unique. To see this, solve the adjoint system with $h(x)={\bar h}$ and initial condition $(\lambda_0,0)$ at $x=-l/2$, where the parameter $\lambda_0$ takes values in the interval $(0,\lambda_s)$:
\begin{equation}\label{opla}
\begin{pmatrix}
\lambda_1(x)\\
\lambda_2(x)
\end{pmatrix}
=
\begin{pmatrix}
-\lambda_s\sech \left( \sqrt{{\bar h}+1}\beta \right)\sinh\left(\sqrt{{\bar h}+1}(x+l/2-\beta) \right)\\
\frac{\lambda_s}{\sqrt{{\bar h}+1}}\left[
\sech \left(\sqrt{{\bar h}+1}\beta \right)\cosh\left(\sqrt{{\bar h}+1}(x+l/2-\beta) \right) 
-1\right]
\end{pmatrix},
\end{equation}
where $\beta$ is uniquely defined by
\begin{equation}\label{betala}
\tanh \left( \sqrt{{\bar h}+1}\beta\right)=\frac{\lambda_0}{\lambda_s}
\end{equation}

Let $T>-l/2$ denote the time that the solution reaches the $\lambda_1$-axis 
again. Then $\lambda_2(T)=0$, and thus
$$
\cosh \left(\sqrt{{\bar h}+1}(T+l/2-\beta) \right)=\cosh\left( \sqrt{{\bar h}+1}\beta\right)
$$
or writing $T$ explicitly as a function of $\lambda_0$ using $(\ref{betala})$:
$$
T(\lambda_0)= \frac{2}{\sqrt{{\bar h}+1}}\arctanh \left( 
\frac{\lambda_0}{\lambda_s}
\right)-\frac{l}{2}
$$
Notice that $\lim_{\lambda_0 \rightarrow 0} T=-l/2$, $\lim_{\lambda_0\rightarrow \lambda_s} T=+\infty$ and $T$ is increasing. Hence, there is a unique $\lambda_0^*$ such that 
\begin{equation}\label{lambda-sol}
T(\lambda_0^*)=l/2.
\end{equation}
Plugging $\lambda_0=\lambda_0^*$ in $(\ref{opla})$ and $(\ref{betala})$ yields the unique 
corresponding $(\lambda_1^*(x),\lambda_2^*(x))$  components of the solution of the Hamiltonian system $(\ref{Ham-sys1})-(\ref{Ham-sys4})$ that satisfy the boundary conditions $(\ref{TC})$.

The $(u,v)$ components corresponding to an optimal solution 
of the Hamiltonian system $(\ref{Ham-sys1})-(\ref{Ham-sys4})$ when $h(x)={\bar h}$ for all $x$ in $[-l/2,l/2]$ can now be determined as well. 
Some orbits of $(\ref{Ham-sys1})-(\ref{Ham-sys2})$ are depicted in Figure $\ref{no-switch-state}$. Arguing as was done for the adjoint system, it 
is not hard to show that the only possible solutions of $(\ref{Ham-sys1})-(\ref{Ham-sys2})$ with $h(x)={\bar h}$ satisfying $(\ref{BC})$ 
must be such that the initial condition $(0,v_0)$ at $x=-l/2$ is such that $0<v_0<1/\sqrt{{\bar h}+1}$. This is because if $v_0\leq 0$ or if 
$v_0\geq 1/\sqrt{{\bar h}+1}$, then the boundary condition $(\ref{BC})$ at $x=l/2$ cannot be satisfied. If $0<v_0<1/\sqrt{{\bar h}+1}$, 
the solution is given by:
\begin{equation}\label{opl}
\begin{pmatrix}
u(x)\\
v(x)
\end{pmatrix}
=
\begin{pmatrix}
-\frac{v_0}{\sqrt{{\bar h}+1}\sinh \left(\sqrt{{\bar h}+1}\alpha\right)}\cosh \left[\sqrt{{\bar h}+1}\left(x+l/2-\alpha\right) \right]+\frac{1}{{\bar h}+1}\\
-\frac{v_0}{\sinh \left(\sqrt{{\bar h}+1}\alpha \right)}\sinh \left[\sqrt{{\bar h}+1}(x+l/2-\alpha) \right]
\end{pmatrix},
\end{equation}
where $\alpha$ is uniquely defined  by
\begin{equation}\label{alfa}
\cotanh \left(\sqrt{{\bar h}+1}\alpha\right)=\frac{1}{v_0\sqrt{{\bar h}+1}}
\end{equation}
Let $T_0>-l/2$ be such that $u(T_0)=0$, then by $(\ref{opl})$ and $(\ref{alfa})$
$$
\cosh \left[\sqrt{{\bar h}+1}(T_0+l/2-\alpha) \right]=\cosh \left[\sqrt{{\bar h}+1} \alpha \right],
$$ 
or, since $T_0>-l/2$, that $T_0=2\alpha - l/2$. Using $(\ref{alfa})$ once more,   $T_0$ can be written explicitly as a function of $v_0$:
\begin{equation}\label{T_0-v_0}
T_0(v_0)=\frac{2}{\sqrt{{\bar h}+1}}\arccoth \left( \frac{1}{v_0\sqrt{{\bar h}+1}}\right) -\frac{l}{2}
\end{equation}
Notice that $\lim_{v_0 \rightarrow 0} T_0=-l/2$, $\lim_{v_0\rightarrow 1/\sqrt{{\bar h}+1}} T_0=+\infty$ and $T_0$ is increasing. Hence, there is a unique $v_0^*$ such that $T_0(v_0^*)=l/2$, namely
\begin{equation}\label{u-sol}
v_0^*=\frac{1}{\sqrt{{\bar h}+1}\coth\left(\sqrt{{\bar h}+1}l/2\right)}.
\end{equation}
Plugging $v_0=v_0^*$ in $(\ref{opl})$ and $(\ref{alfa})$ yields the unique 
corresponding $(u^*(x),v^*(x))$  components of the solution of the Hamiltonian system $(\ref{Ham-sys1})-(\ref{Ham-sys4})$ that satisfy the boundary conditions $(\ref{BC})$.
\begin{figure}[ht]
\centering
\includegraphics[width=3in]{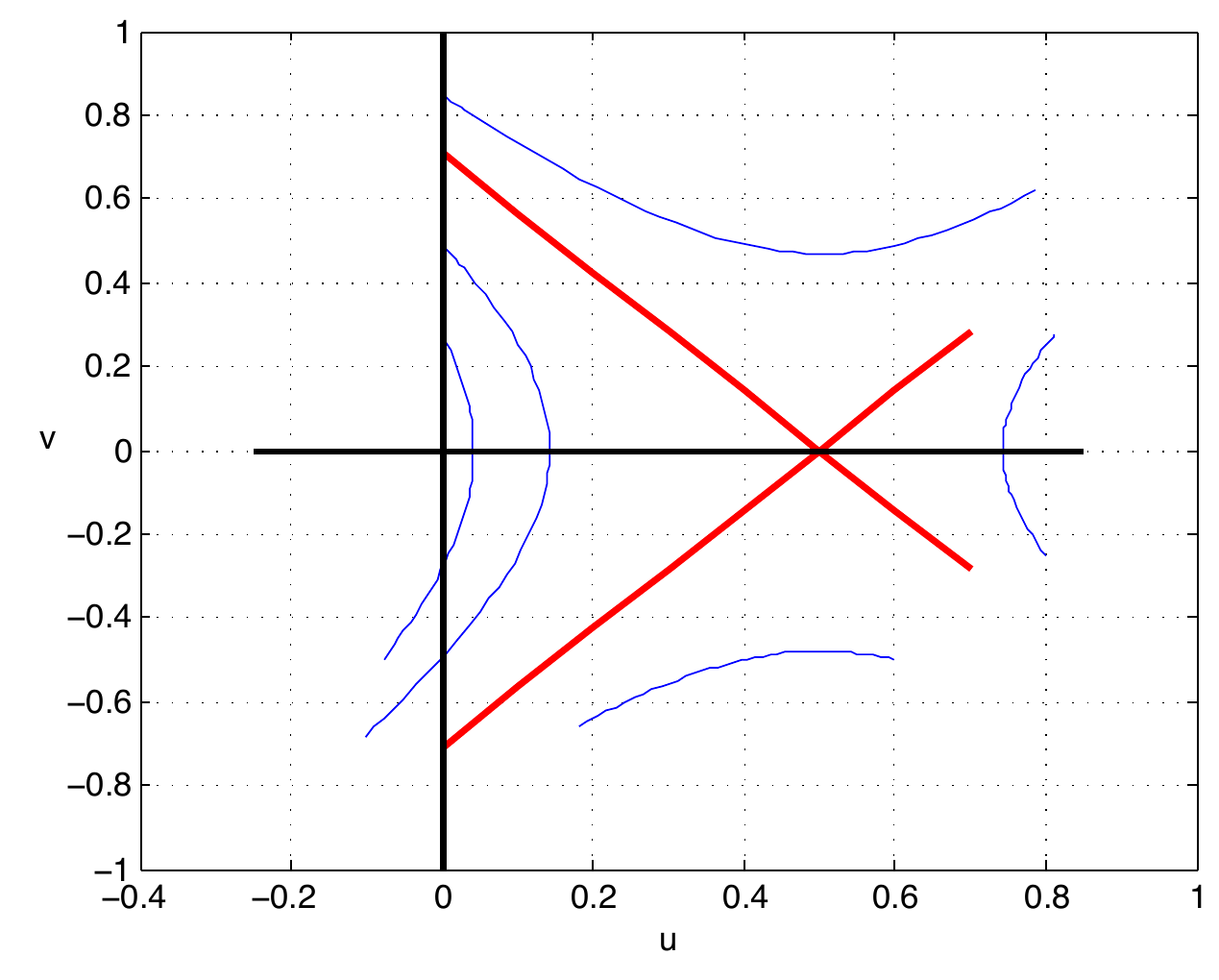}
\caption{Phase portrait of system $(\ref{Ham-sys1})-(\ref{Ham-sys2})$, with $l=2$, $q=0.5$ and $h={\bar h}=1$.  
The stable manifold is the straight line with negative slope (in red), whereas the unstable manifolds is the straight line with positive slope (also in red).}
\label{no-switch-state}
\end{figure}

Summarizing the results obtained in this subsection, and combining them with Theorem $\ref{existence}$, there follows:
\begin{stel} \label{q-less-than-1}
If $0<q\leq 1$, then 
there is a unique optimal control $h^*(x)={\bar h}$ for all $x$ in $[-l/2,l/2]$ which maximizes the scaled functional $(\ref{scaled-cost})$ for the 
steady state problem $(\ref{ss1})-(\ref{BC-state})$.  
The corresponding optimal fish density $u(x)=u^*(x)$ is given by $(\ref{opl})-(\ref{alfa})$ with $v_0=v_0^*$, where $v_0^*$ is defined in $(\ref{u-sol})$.
\end{stel}

\subsection{The case $q>1$.}
Some orbits of the adjoint system $(\ref{Ham-sys3})-(\ref{Ham-sys4})$ are depicted in Figure $\ref{adjoint-switch}$, using $h(x)={\bar h}$ on the left and $h(x)=0$ on the right. In view of $(\ref{optimal2})$, solutions of the adjoint system follow orbits of the left panel with 
$h(x)={\bar h}$ as long as $\lambda_2>-1/l$, and those of the right with $h(x)=0$  whenever $\lambda_2<-1/l$. 
\begin{figure}[ht]
\centering
\includegraphics[width=5in]{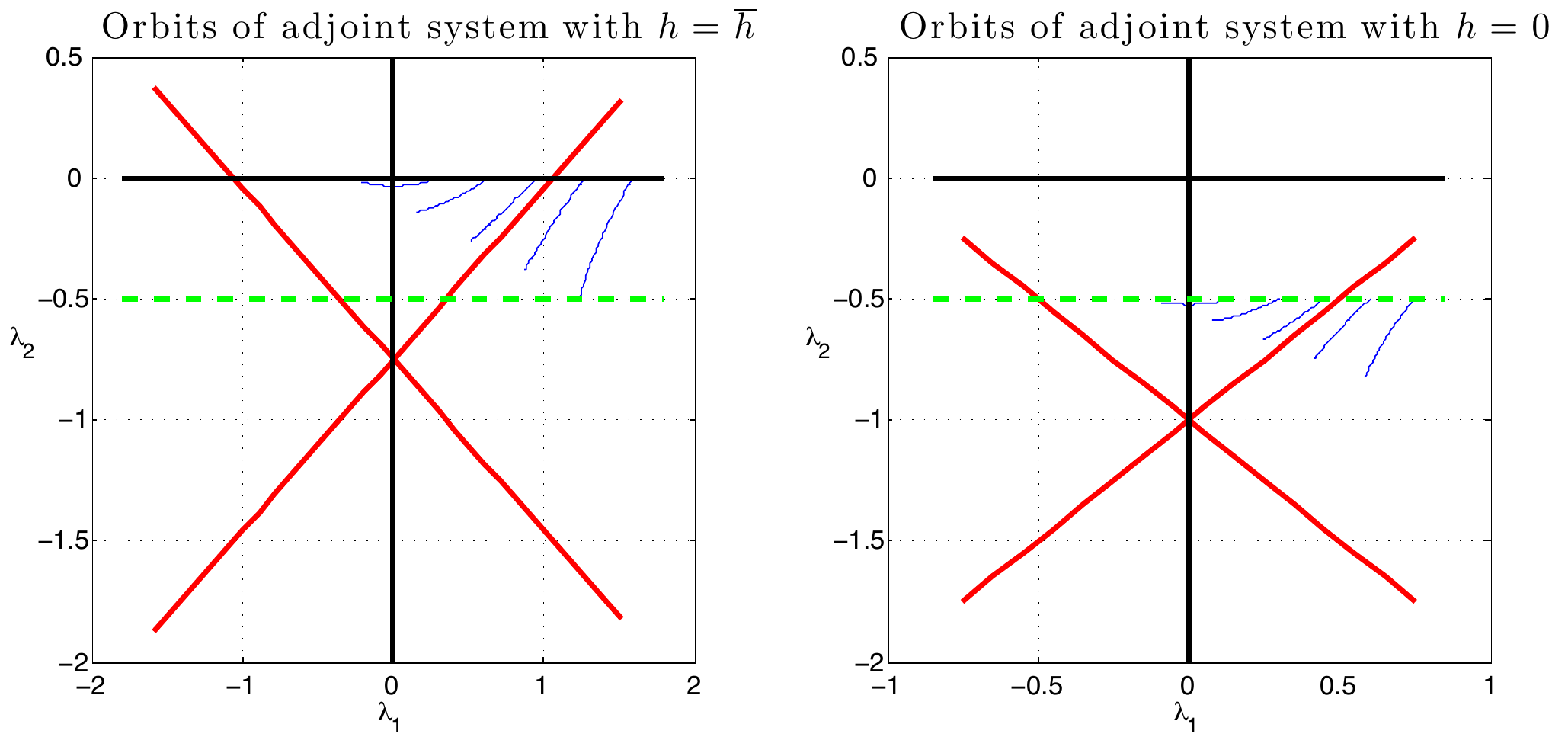}
\caption{Phase portrait of system $(\ref{Ham-sys3})-(\ref{Ham-sys4})$, with $l=2$, $q=2$ and $h={\bar h}=1$ on the left, and 
$h=0$ on the right. Stable manifolds are the straight lines with positive slope (in red), whereas unstable manifolds are the straight lines  with negative slope (also in red). The dashed curve (in green) is the switching line where $\lambda_2=-1/l$.}
\label{adjoint-switch}
\end{figure}
Defining $F_1$ and $F_2$ as the (autonomous) vector field of  
the adjoint system $(\ref{Ham-sys3})-(\ref{Ham-sys4})$ when $h(x)={\bar h}$ and $h(x)=0$ respectively, the adjoint system can be rewritten as an autonomous system:
\begin{equation}\label{ad-feedback}
{\dot \lambda}=
\begin{cases}
F_1(\lambda),\textrm{ if } \lambda \in {\cal S}_a\\
F_2(\lambda), \textrm{ if } \lambda \in {\cal S}_b
\end{cases}
\end{equation}
where ${\cal S}_a:=\{(\lambda_1,\lambda_2)|\lambda_2>-1/l\}$ and 
${\cal S}_b:=\{(\lambda_1,\lambda_2)|\lambda_2<-1/l\}$ are the regions above and below the switching line respectively.
The objective is to find  solutions of $(\ref{ad-feedback})$ that satisfy $(\ref{TC})$. 
Actually, it will shortly become clear that thanks to a symmetry property of the adjoint system, it suffices to consider solutions solutions defined on 
just $[-l/2,0]$ (half of the control horizon $[-l/2,l/2]$) that start on the $\lambda_1$-axis at $x=-l/2$ and end on the $\lambda_2$-axis at $x=0$. 

Some notation is introduced first. Let $R:(\lambda_1,\lambda_2)\longrightarrow (-\lambda_1,\lambda_2)$ denote 
the reflection with respect to the $\lambda_2$-axis.  
\begin{lemma}\label{sym}
For  $i=1,2$, there holds that:
\begin{equation}\label{symmetry}
F_i\circ R=-R\circ F_i
\end{equation}
The regions ${\cal S}_a$ and ${\cal S}_b$ are invariant under $R$. 
Consequently, if $(\lambda_1(x),\lambda_2(x))$ with $x\in [a,b]$ is a solution of the adjoint system $(\ref{ad-feedback})$, then so is 
$R(\lambda_1(-x),\lambda_2(-x))$ with $x\in [-b,-a]$.
\end{lemma}
\begin{proof}
The verification of $(\ref{symmetry})$ is straightforward, and so is the invariance of ${\cal S}_a$ and ${\cal S}_b$ under $R$. 
Consider a solution $(\lambda_1(x),\lambda_2(x))$ of the adjoint system $(\ref{ad-feedback})$, defined on some interval $[a,b]$. 
By breaking up $[a,b]$ into smaller subintervals if necessary, it may be assumed without loss of generality that the solution remains 
entirely in one of the (closures of the) regions ${\cal S}_a$ or ${\cal S}_b$, where the system's vector field is given by $F_1$ or $F_2$ 
respectively. By invariance of the regions ${\cal S}_a$ 
and ${\cal S}_b$ under the reflection $R$, the function $({\tilde \lambda_1}(x),{\tilde \lambda_2}(x)):=R(\lambda_1(-x),\lambda_2(-x))$, $x\in [-b,-a]$ lies in the same region as $(\lambda_1(x),\lambda_2(x))$. Finally, it is claimed that this function is also a solution of the 
adjoint system $(\ref{ad-feedback})$. 
Below,  $F_i$ where $i=1$ or $2$, denotes the vector field of the adjoint system that corresponds to the considered solution. 
Using $(\ref{symmetry})$,
\begin{equation*}
\begin{pmatrix}
\frac{d{\tilde \lambda_1}}{dx}(x)\\
\frac{d{\tilde \lambda_1}}{dx}(x)
\end{pmatrix}=
-DR\circ \begin{pmatrix}
\frac{d \lambda_1}{dx}(-x)\\
\frac{d \lambda_2}{dx}(-x)
\end{pmatrix}
=-R\circ F_i
\begin{pmatrix}
\lambda_1(-x)\\
\lambda_2(-x)
\end{pmatrix}
=
F_i\circ R
\begin{pmatrix}
\lambda_1(-x)\\
\lambda_2(-x)
\end{pmatrix}
=
F_i
\begin{pmatrix}
{\tilde \lambda_1}(x)\\
{\tilde \lambda_2}(x)
\end{pmatrix}\textrm{ for } x\in [-b,-a],
\end{equation*}
establishing the claim.
\end{proof}
From Lemma $\ref{sym}$ follows:
\begin{gevolg}\label{halvering}
Let $\lambda(x)$, $x\in [-l/2,l/2]$ be a solution of $(\ref{ad-feedback})$ satisfying $(\ref{TC})$. Then 
\begin{equation}\label{TC2}
\lambda_2(-l/2)=\lambda_1(0)=0
\end{equation}
Conversely, if $\lambda(x)$, $x\in [-l/2,0]$ is a solution of $(\ref{ad-feedback})$ satisfying $(\ref{TC2})$, then
$$
\lambda_e(x):=\begin{cases}
\lambda(x),\textrm{ if } x\in[-l/2,0]\\
R(\lambda(-x)),\textrm{ if }x\in [0,l/2]
\end{cases}
$$
is a solution of $(\ref{ad-feedback})$ satisfying $(\ref{TC})$.
\end{gevolg}
\begin{proof}
Let $\lambda(x)$, $x\in [-l/2,l/2]$ solve $(\ref{ad-feedback})$ and satisfy $(\ref{TC})$. Then necessarily $\lambda_1(-l/2)>0$, since any solution of $(\ref{ad-feedback})$ starting in an initial condition with $\lambda_1(-l/2)\leq 0$, enters the interior of the second quadrant, and remains there for all future times so that $(\ref{TC})$ cannot hold at $x=l/2$. A similar argument shows that necessarily $\lambda_1(l/2)<0$ as well.  Then by continuity of $\lambda(x)$, it follows that there is some $x^*\in (-l/2,l/2)$, such that:
\begin{equation}\label{hulp}
\lambda_1(x^*)=0
\end{equation}
By Lemma $\ref{sym}$, the function ${\tilde \lambda}(x):=R(\lambda(-x))$, $x\in [-l/2,l/2]$ is also a solution of $(\ref{ad-feedback})$. Notice that 
by $(\ref{hulp})$ and since $R$ leaves the $\lambda_2$-axis invariant,
$$
{\tilde \lambda}(-x^*)=R(\lambda(x^*))=\lambda(x^*)
$$
Therefore, $\lambda(x)$ and ${\tilde \lambda}(x)$ are solutions of the same autonomous system $(\ref{ad-feedback})$ with 
coinciding states at $x=x^*$ and $x=-x^*$ respectively. Hence, by uniqueness of solutions, these solutions must agree up to a shift in $x$:
$$
\lambda(x)={\tilde \lambda} (x-2x^*)
$$
It is claimed that in fact
$$
x^*=0.
$$
If this were not the case then either $x^*>0$ or $x^*<0$. Assume first that $x^*>0$. Then $\lambda(l/2)={\tilde \lambda}(l/2-2x^*)$ belongs to 
the negative $\lambda_1$-axis by $(\ref{TC})$. But since $x^*>0$, and since forward solutions starting on the negative $\lambda_1$-axis enter and remain in the interior of the second quadrant, ${\tilde \lambda}(l/2)={\tilde \lambda}(2x^*+(l/2-2x^*))$ belongs to the interior of the second quadrant. However, 
${\tilde \lambda}(l/2)=R(\lambda (-l/2))$ belongs to the negative $\lambda_1$-axis, since $\lambda(-l/2)$ belongs to the positive $\lambda_1$-axis. Thus, a contradiction has been reached. A similar argument shows that $x^*$ cannot be negative. Thus, $x^*=0$ and 
$$
\lambda(x)={\tilde \lambda}(x),\textrm{ for all } x \in [-l/2,l/2],
$$
and $(\ref{hulp})$ specializes to $\lambda_1(0)=0$, which establishes the second equation in $(\ref{TC2})$.

To prove the converse, note that if $\lambda(x)$, $x$ in $[-l/2,0]$, solves $(\ref{ad-feedback})$, then $R(\lambda(-x))$, $x$ in $[0,l/2]$, 
solves $(\ref{ad-feedback})$ as well by Lemma $\ref{sym}$. Moreover, since $\lambda(0)$ belongs to the $\lambda_2$-axis by $(\ref{TC2})$, 
it equals $R(\lambda (-0))$, and thus $\lambda_e(x)$, $x$ in $[-l/2,l,2]$, is a continuous solution of $(\ref{ad-feedback})$. Finally, since 
$\lambda_e(l/2)=R(\lambda(-l/2))$ belongs to the $\lambda_1$-axis, it satisfies $(\ref{TC})$ at $x=l/2$, which concludes the proof.
\end{proof}
From Corollary $\ref{halvering}$ follows that  solutions of $(\ref{ad-feedback})$ satisfying $(\ref{TC2})$ should be determined. 
A brief sketch to approach this problem is given next. 
Consider all solutions of $(\ref{ad-feedback})$ starting on the positive $\lambda_1$-axis at $x=0$ (rather than at $x=-l/2$; this is easily achieved by a shift in $x$), by parametrizing them by 
the $\lambda_1$-coordinate  of their initial condition. Concerning these solutions, focus on the following two questions:
\begin{enumerate}
\item
Which solutions reach the $\lambda_2$-axis?
\item
For those solutions reaching the $\lambda_2$-axis, what is the first $x$-value larger than $0$ for which this happens.
\end{enumerate}
With respect to the first question, it will be seen that 
only some solutions, namely those with corresponding parameter values that are not too high, reach the $\lambda_2$-axis.
Moreover, some -but not all- of these solutions cross the switching line. With respect to the second question, it will be seen that 
the first value of $x$ for which solutions reach the $\lambda_2$-axis, is an increasing function of the 
parameter. It increases from zero to infinity, and hence there will be a unique solution for which it equals $l/2$. The 
solution corresponding to this $x$-value, is the sought-after solution to problem $(\ref{ad-feedback})$ with $(\ref{TC2})$.

Thus,  consider
\begin{equation}\label{shoot}
{\dot \lambda}=
\begin{cases}
F_1(\lambda),\textrm{ if } \lambda \in {\cal S}_a\\
F_2(\lambda), \textrm{ if } \lambda \in {\cal S}_b
\end{cases}, 
\begin{pmatrix}
\lambda_1(0)\\
\lambda_2(0)
\end{pmatrix}=
\begin{pmatrix}
\lambda_0\\
0
\end{pmatrix}
\end{equation}
where $\lambda_0>0$ is a parameter. To avoid cumbersome notation in the subsequent calculations,  define 
the positive constants
\begin{equation}\label{a-en-b}
a_1={\bar h}+1,\;\; b_1={\bar h}+q\textrm{ and }a_2=1,\;\; b_2=q,
\end{equation}
so that the vector fields $F_i$, $i=1,2$, can be rewritten as
$$
F_i(\lambda_1,\lambda_2)=\begin{pmatrix}
-a_i\lambda_2-\frac{b_i}{l}\\
-\lambda_1
\end{pmatrix}
$$
Also define $\lambda_1$-coordinates of the intercepts of the stable manifolds of the adjoint system $(\ref{Ham-sys3})-(\ref{Ham-sys4})$ with $h(x)={\bar h}$, and with $h(x)=0$ and the $\lambda_1$-axis as
\begin{equation}\label{i1-en-i2}
i_1:=\frac{b_1}{\sqrt{a_1}l}\textrm{ and }i_2:=\frac{b_2}{\sqrt{a_2}l}
\end{equation}
respectively.
Since $q>1$, it can be verified by simple calculations that
\begin{equation}\label{twee-ineq}
e_1:=-\frac{b_1}{a_1l}>-\frac{b_2}{a_2l}=:e_2,\;\; \textrm{ and }\;\; i_{s_1}:=\frac{\sqrt{a_1}}{l}\left(\frac{b_1}{a_1}-1 \right)<\frac{\sqrt{a_2}}{l}\left(\frac{b_2}{a_2}-1 \right)=:i_{s_2}
\end{equation}
The first inequality in $(\ref{twee-ineq})$ expresses that the $\lambda_2$-coordinate of the equilibrium point of the 
adjoint system $(\ref{Ham-sys3})-(\ref{Ham-sys4})$ with $h(x)={\bar h}$ is larger than the  $\lambda_2$-coordinate of the equilibrium point of the adjoint system $(\ref{Ham-sys3})-(\ref{Ham-sys4})$ with $h(x)=0$. The second inequality expresses that the $\lambda_1$-coordinate of the intersection of the switching line where $\lambda_2=-1/l$ and the stable manifold of the equilibrium point of the adjoint system $(\ref{Ham-sys3})-(\ref{Ham-sys4})$ with $h(x)={\bar h}$, is smaller than 
the $\lambda_1$-coordinate of the stable manifold of the equilibrium point of the adjoint system $(\ref{Ham-sys3})-(\ref{Ham-sys4})$ 
with $h(x)=0$. These geometrical observations turn out to be important for the subsequent calculations. They  are illustrated in 
Figure $\ref{adjoint-switch}$.

Define two important values for the parameter $\lambda_0$:
\begin{equation}\label{grenzen}
\lambda_0^*=\frac{\sqrt{2b_1-a_1}}{l}\textrm{ and } \lambda_0^{**}=\left((\lambda_0^*)^2+i_{s,2}^2\right)^{1/2}
\end{equation}
Let $\lambda(x)$, $x>0$, be the (forward) solution of $(\ref{shoot})$, and define
\begin{equation}\label{time-def}
T_0(\lambda_0)=\inf \{x>0|\lambda_1(x)=0\},
\end{equation}
the first instance where $\lambda(x)$ hits the $\lambda_2$-axis. If $\lambda(x)$ never hits the $\lambda_2$-axis, then set $T(\lambda_0)=+\infty$. Since the system is piecewise linear, it can be solved  analytically, and calculate $T_0$ explicitly:
\begin{stel}\label{time-exp}
The function $T_0:(0,+\infty)\rightarrow (0,+\infty]$ is given by:
$$
T_0(\lambda_0)=
\begin{cases}
\frac{1}{\sqrt{a_1}}\arctanh \left(\frac{\lambda_0}{i_1} \right),\textrm{ if  } 0<\lambda_0<\lambda_0^*\\
\frac{1}{\sqrt{a_1}}\left(\arctanh\left(\frac{\lambda_0}{i_1}\right)-\arccosh \left( 
\frac{b_1/a_1-1}{b_1/a_1}\frac{1}{\sqrt{1-\left(\frac{\lambda_0}{i_1}\right)^2}}\right) \right)+
\frac{1}{\sqrt{a_2}}\arctanh \left(\frac{\sqrt{\lambda_0^2-(\lambda_0^*)^2}}{i_{s2}} \right),\textrm{ if }\lambda_0^*\leq \lambda_0<i_1\\
-\frac{1}{\sqrt{a_1}}\ln \left( \frac{b_1/a_1-1}{b_1/a_1}\right)+
\frac{1}{\sqrt{a_2}}\arctanh \left(\frac{\sqrt{\lambda_0^2-(\lambda_0^*)^2}}{i_{s2}} \right),
\textrm{ if }\lambda_0=i_1\\
\frac{1}{\sqrt{a_1}}\left(\arccoth\left(\frac{\lambda_0}{i_1}\right)-\arcsinh \left( 
\frac{b_1/a_1-1}{b_1/a_1}\frac{1}{\sqrt{\left(\frac{\lambda_0}{i_1}\right)^2-1}}\right) \right)+
\frac{1}{\sqrt{a_2}}\arctanh \left(\frac{\sqrt{\lambda_0^2-(\lambda_0^*)^2}}{i_{s2}} \right),
\textrm{ if } i_1<\lambda_0<\lambda_0^{**}\\
+\infty,\textrm{ if  }\lambda_0^{**}\leq \lambda_0
\end{cases}
$$
$T_0$ is continuous and increasing, and 
\begin{equation}\label{lims}
\lim_{\lambda_0\rightarrow 0}T_0(\lambda_0)=0\textrm{ and }\lim_{\lambda_0\rightarrow \lambda_0^{**}}T_0(\lambda_0)=+\infty
\end{equation}
There exists a unique ${\bar \lambda_0} \in (0,\lambda_0^{**})$ such that
\begin{equation}\label{lambda-bar}
T_0({\bar \lambda_0})=\frac{l}{2}
\end{equation}
with corresponding solution $\lambda(x)$ of $(\ref{shoot})$, satisfying
\begin{equation}\label{TC3}
\lambda_2(0)=\lambda_1(l/2)=0
\end{equation}
This solution hits the switching line where $\lambda_2=-1/l$ for some $x\in (0,l/2)$ if and only if 
$$
\lambda_0^*<{\bar \lambda_0},
$$
or equivalently, by applying the increasing function $T_0$, if and only if
\begin{equation}\label{criterion}
\frac{1}{\sqrt{a_1}}\arctanh \left(\frac{\lambda_0^*}{i_1} \right)<\frac{l}{2}
\end{equation}
\end{stel}
\begin{proof}
Theorem $\ref{time-exp}$ is proved in the Appendix A.
\end{proof}
Inequality $(\ref{criterion})$ is of great practical relevance. It determines the minimal coastal length 
\begin{equation}\label{min-length}
l_{\min}=\frac{2}{\sqrt{a_1}}\arctanh \left(\frac{\lambda_0^*}{i_1} \right)=\frac{2}{\sqrt{{\bar h}+1}}\arctanh 
\left(\frac{\sqrt{({\bar h}+1)({\bar h}+2q-1)}}{{\bar h}+q} \right)
\end{equation}
that is required in order for the implementation of an MPA to be optimal. 
For coastlines with a length below $l_{\min}$, MPA's should not be installed, and fishing should take place with maximal harvesting rate 
${\bar h}$ everywhere along the coast. For coastlines longer than $l_{\min}$, the optimal solution requires the placement of an MPA. Where this should occur is addressed in the next result
\begin{stel}\label{location}
Assume that $l>l_{\min}$, and consider the unique value ${\bar \lambda_0}\in (\lambda_0^*,\lambda_0^{**})$ defined in $(\ref{lambda-bar})$. 
Denote the corresponding solution of $(\ref{shoot})$ by $\lambda(x)$, and let $T_s({\bar \lambda_0})$ be the $x$-value at which the 
solution $\lambda(x)$ hits the switching line where $\lambda_2=-1/l$. Then
\begin{equation}\label{where}
T_s({\bar \lambda_0})=
\begin{cases}
\frac{1}{\sqrt{a_1}}\left( \arctanh \left(\frac{{\bar \lambda_0}}{i_1} \right) -
\arccosh \left(\frac{b_1/a_1-1}{b_1/a_1}\frac{1}{\sqrt{1-\left(\frac{{\bar \lambda_0}}{i_1} \right)^2}} \right)
\right)\textrm{ if } \lambda_0^*<{\bar \lambda_0}<i_1\\
-\frac{1}{\sqrt{a_1}}\ln \left(\frac{b_1/a_1-1}{b_1/a_1} \right)\textrm{ if } {\bar \lambda_0}=i_1\\
\frac{1}{\sqrt{a_1}}\left( \arccoth \left(\frac{{\bar \lambda_0}}{i_1} \right)
-\arcsinh \left(\frac{b_1/a_1-1}{b_1/a_1}\frac{1}{\sqrt{\left(\frac{{\bar \lambda_0}}{i_1} \right)^2-1}} \right)
\right)\textrm{ if } i_1<{\bar \lambda_0}<\lambda_0^{**}
\end{cases}
\end{equation}
\end{stel}
Theorem $\ref{location}$ is proved in the Appendix B.

Combining Theorems $\ref{time-exp}$ and $\ref{location}$ yields the main result.
\begin{stel} \label{q>1}
Assume that $q> 1$. 
\begin{itemize}
\item
If $l\leq l_{\min}$,  then 
there is a unique optimal control $h^*(x)={\bar h}$ for all $x$ in $[-l/2,l/2]$ which maximizes the scaled functional $(\ref{scaled-cost})$ for the 
steady state problem $(\ref{ss1})-(\ref{BC-state})$.  The corresponding optimal fish density $u(x)=u^*(x)$ is given by $(\ref{opl})-(\ref{alfa})$ with $v_0=v_0^*$, where $v_0^*$ is defined in $(\ref{u-sol})$.
\item
If $l>l_{\min}$, then there is a unique optimal control
$$
h^*(x)=\begin{cases}
{\bar h}\textrm{ if } x\in \left[-l/2,-l/2+T_s({\bar \lambda_0})\right)\textrm{ or }x\in \left(l/2-T_s({\bar \lambda_0}),l/2\right]\\
0,\textrm{ if }x\in [-l/2+T_s({\bar \lambda_0}),l/2-T_s({\bar \lambda_0})]\\
\end{cases}
$$
where $T_s({\bar \lambda_0})$ is defined in $(\ref{where})$. This optimal control maximizes the scaled functional $(\ref{scaled-cost})$ for the 
steady state problem $(\ref{ss1})-(\ref{BC-state})$.  There is a corresponding optimal fish density $u^*(x)$, defined as 
the $u$-component of the unique solution to $(\ref{ss1})-(\ref{BC-state})$, with $h(x)=h^*(x)$.
\end{itemize}
\end{stel}
\begin{proof}
If $l\leq l_{\min}$, then it follows from Theorem $\ref{time-exp}$ that the optimal control never exhibits a switch during the 
control horizon $[-l/2,l/2]$, and thus 
$h^*(x)={\bar h}$ for all $x\in [-l/2,l/2]$. The corresponding optimal fish density then follows as in the case where $0<q\leq 1$, see 
Theorem $\ref{q-less-than-1}$.

Now assume that $l>l_{\min}$. The form of an optimal control follows from Theorem $\ref{time-exp}$ for $x\in [-l/2,0]$, and by the symmetry 
of the adjoint system (see Corollary $\ref{halvering}$) for $x\in [0,l/2]$. What remains to be proved is that the steady state 
problem $(\ref{ss1})-(\ref{BC-state})$ with $h(x)=h^*(x)$ has a unique solution.
First, notice that system $(\ref{ss1})-(\ref{ss2})$ with $h(x)=h^*(x)$ is a strongly monotone system \cite{smith} because the Jacobian matrix
$$
\begin{pmatrix}
0&1\\
1+h^*(x)&0
\end{pmatrix}
$$
has positive off-diagonal entries, and because it is an irreducible matrix for all $x\in [-l/2,l/2]$.
It is claimed that this implies that $(\ref{ss1})-(\ref{BC-state})$ with $h(x)=h^*(x)$, has a unique solution. Indeed, assume that 
$(0,v_1(-l/2))$ and $(0,v_2(-l/2))$ were two distinct initial conditions at $x=-l/2$ whose corresponding solutions $(u_1(x),v_1(x))$ 
and $(u_2(x),v_2(x))$ would solve  $(\ref{ss1})-(\ref{BC-state})$, with $h(x)=h^*(x)$. Assume without loss of generality that $v_1<v_2$ (otherwise, simply relabel the initial conditions). By strong monotonicity, there holds that:
$$
\begin{pmatrix}
u_1(x)\\
v_1(x)
\end{pmatrix}<<
\begin{pmatrix}
u_2(x)\\
v_2(x)
\end{pmatrix}
$$
for all $x\in (-l/2,l/2]$. Here $<<$ means that both components of the two vectors are strictly ordered according to the usual strict order $<$ 
on the real numbers.
This contradicts in particular that $u_1(l/2)=u_2(l/2)=0$.
\end{proof}
\begin{opm}
Theorems $\ref{q-less-than-1}$ and $\ref{q>1}$ establish the optimal controls for the steady state problem associated to 
system $(\ref{scaled})$ with scaled cost $(\ref{scaled-cost})$. A natural question is whether the steady state corresponding 
to this optimal control $h^*(x)$ is asymptotically stable for the time problem $(\ref{scaled})$. Here we briefly show that the 
eigenvalues of the linearization of system $(\ref{scaled})$ at the steady state $u^*(x)$ corresponding to $h^*(x)$, are always negative, 
providing evidence of asymptotic stability. Linearization leads to the following eigenvalue problem:
\begin{eqnarray}
\lambda w&=&w_{xx}-(1+h^*(x))w\nonumber \\
w(-l/2)&=&w(l/2)=0 \label{eig-problem}
\end{eqnarray}
An integration by parts shows that the operator $L[w]:=w_{xx}-(1+h^*(x))w$ is self-adjoint for 
the continuously differentiable functions on $[-l/2,l/2]$ that satisfy the 
Dirichlet boundary condition at $x=\pm l/2$, i.e. $(L[w_1],w_2)=(w_1,L[w_2])$, where $(w_1,w_2):=\int_{-l/2}^{l/2}w_1w_2dx$ denotes the 
inner product on this space. Consequently, all eigenvalues $\lambda$ are real. If $(\lambda,w(x))$ is an eigenvalue-eigenfunction pair, 
i.e. if $\lambda$ is a real number and $w(x)$ a corresponding eigenfunction, i.e. a non-zero solution of $(\ref{eig-problem})$, then multiplying the equation in 
$(\ref{eig-problem})$ by $w$, integrating by parts over $[-l/2,l/2]$, and using the boundary condition, yields:
$$
\lambda \int_{-l/2}^{l/2}w^2dx=-\int_{-l/2}^{l/2}w_x^2+(1+h^*(x))w^2dx,
$$
from which follows that $\lambda<0$, as claimed.
\end{opm}
\section{Discussion}

In this paper  the problem of when and where to implement an MPA or network of MPA's along a fixed coastline was considered. 
Specifically, the question of maximizing the functional $(\ref{unscaled-cost})$ for the steady state corresponding to 
problem $(\ref{unscaled})$, over all harvesting rate functions $H(X)$ taking values in $[0,{\bar H}]$, was considered. It was found that the optimal solution 
had a particularly simple structure: either no, or just a single Marine Reserve where fishing is prohibited, should be established. 
\begin{table}
\centering
\begin{tabular}{|c|c|c|}
\hline
${\bf L\backslash Q}$ & $Q\leq \mu$   & $Q>\mu$      \\
\hline
$L\leq L_{\min}$    & no MPA      & no MPA   \\
\hline
$L> L_{\min}$    & no MPA      & $1$ MPA   \\
\hline
\end{tabular}
\caption{Summary of main result}
\label{tabel}
\end{table}
The main results -which are succinctly summarized in Table $\ref{tabel}$, and where $L_{\min}$ is defined below- were proved for a scaled model, yet they are easily phrased in terms of the parameters of the original, unscaled model:
\begin{enumerate}
\item If $Q\leq \mu$, then the optimal harvesting strategy is to fish at maximal rate ${\bar H}$ everywhere along the coastline. 
No MPA should be installed in this case.
\item If $Q>\mu$, then the optimal harvesting strategy depends on the length of the coastline. Letting
\begin{equation}\label{unscaled-length}
L_{\min}=
2\sqrt{\frac{D}{{\bar H}+\mu}}\arctanh 
\left(\frac{\sqrt{({\bar H}+\mu)({\bar H}+2Q-\mu)}}{{\bar H}+Q} \right),
\end{equation}
there holds that:
\begin{enumerate}
\item if $L\leq L_{\min}$, then  the optimal harvesting strategy is still to fish at maximal rate ${\bar H}$ everywhere; no MPA should be implemented in this case either.
\item if $L>L_{\min}$, then  the optimal harvesting strategy requires the installation of a single MPA in the middle of the 
coastline; the MPA takes the form of a Marine Reserve where fishing is prohibited. Outside the reserve, fishing is allowed, and occurs at maximal rate ${\bar H}$: 
$$
H(X)=\begin{cases}
{\bar H}\textrm{ if } X\in [-L/2,-B) \textrm{ or } X\in (B,L/2]\\
0\textrm{ if } X\in[-B,B]
\end{cases}
$$
The Marine Reserve boundary $B$ is determined as follows:
\begin{eqnarray*}
B=\begin{cases}
\frac{L}{2}-\sqrt{\frac{D}{{\bar H}+\mu}}\left( \arctanh(\lambda^*)-
\arccosh\left(\frac{Q-\mu}{Q+{\bar H}}\frac{1}{\sqrt{1-(\lambda^*)^2}} \right)\right),\textrm{ if }
\frac{\sqrt{({\bar H}+2Q-\mu)({\bar H}+\mu)}}{{\bar H}+Q}<\lambda^*<1\\
\frac{L}{2}+\sqrt{\frac{D}{{\bar H}+\mu}}\ln \left( \frac{Q-\mu}{Q+{\bar H}}\right),\textrm{ if } \lambda^*=1\\
\frac{L}{2}-\sqrt{\frac{D}{{\bar H}+\mu}}\left(\arccoth(\lambda^*)-\arcsinh \left(\frac{Q-\mu}{Q+{\bar H}}
\frac{1}{\sqrt{(\lambda^*)^2-1}} \right)\right),\textrm{ if } 1<\lambda^*<\frac{\sqrt{({\bar H}+Q^2/\mu)({\bar H}+\mu)}}{{\bar H}+Q}
\end{cases}
\end{eqnarray*}
where $\lambda^*$ is the unique solution of the equation
\begin{equation}\label{inversie}
F(\lambda^*)=\frac{L}{2},
\end{equation}
and the increasing function $F$ is defined as
\end{enumerate}
\end{enumerate}
$$
F(\lambda)=
\begin{cases}
\sqrt{\frac{D}{{\bar H}+\mu}}\left( \arctanh(\lambda)-
\arccosh\left(\frac{Q-\mu}{Q+{\bar H}}\frac{1}{\sqrt{1-\lambda^2}} \right)\right)+
\sqrt{\frac{D}{\mu}}
\arctanh\left(\frac{{\bar H}+Q}{Q-\mu}\sqrt{\frac{\mu}{{\bar H}+\mu}}\sqrt{\lambda^2-\frac{({\bar H}+2Q-\mu)({\bar H}+\mu)}{({\bar H}+Q)^2}} \right),\\
\textrm{ if }\frac{\sqrt{({\bar H}+2Q-\mu)({\bar H}+\mu)}}{{\bar H}+Q}<\lambda<1,\\
-\sqrt{\frac{D}{{\bar H}+\mu}}\ln \left( \frac{Q-\mu}{Q+{\bar H}}\right)+\sqrt{\frac{D}{\mu}}
\arctanh\left(\frac{1}{Q-\mu}\sqrt{\frac{\mu}{{\bar H}+\mu}}\sqrt{({\bar H}+Q)^2-({\bar H}+2Q-\mu)({\bar H}+\mu)} \right),\textrm{ if } \lambda=1\\
\sqrt{\frac{D}{{\bar H}+\mu}}\left(\arccoth(\lambda)-\arcsinh \left(\frac{Q-\mu}{Q+{\bar H}}
\frac{1}{\sqrt{\lambda^2-1}} \right)\right)+
\sqrt{\frac{D}{\mu}}
\arctanh\left(\frac{{\bar H}+Q}{Q-\mu}\sqrt{\frac{\mu}{{\bar H}+\mu}}\sqrt{\lambda^2-\frac{({\bar H}+2Q-\mu)({\bar H}+\mu)}{({\bar H}+Q)^2}} \right),\\
\textrm{ if } 1<\lambda<\frac{\sqrt{({\bar H}+Q^2/\mu)({\bar H}+\mu)}}{{\bar H}+Q}
\end{cases}
$$
Recall that $Q$ is the weight parameter in the functional  of the term representing the average fish density along the coastline, and as mentioned before that it is small if conservationists' pressure is small (or if fisheries interests are dominant), and large otherwise. The results indicate that if this weight 
is too small, namely if $Q\leq \mu$, then no MPA should be established, regardless of the size of the coast. If on the other hand, this weight is large enough, namely if 
$Q>\mu$, then a single MPA taking the form of a Marine Reserve should be established provided that the coast is long enough. Precisely how long is determined by 
the minimal length $L_{\min}$ in $(\ref{unscaled-length})$. The existence of a minimal coast length that warrants the implementation of an MPA, is somewhat reminiscent of the existence of a minimal patch size which guarantees species persistence
in ecological models that take the form of reaction-diffusion equations \cite{cosner}.

The formula for $L_{\min}$ reveals that more mobile species (having larger $D$, and assuming that all 
other parameters remain fixed) require a larger minimal length to warrant the installation of an MPA. 
Similarly, larger $Q$ but keeping all other parameters fixed, leads to lower $L_{\min}$. In other words, if pressure by conservationists increases, then MPA's should be established for shorter coastal lengths. The investigation of how $L_{\min}$ depends on the other 
parameters ${\bar H}$ and $\mu$, can be done similarly using formula $(\ref{unscaled-length})$, by calculating the appropriate derivatives. 
In case an MPA is required, the results also indicate where it should be placed, namely in the middle of the coastline, 
and what size it should be, namely $2B$. A (admittedly ugly) formula for $B$ is available, and only implicitly, since it is based on 
the value $\lambda^*$ from $(\ref{inversie})$, which requires the inversion of the function $F$. Nevertheless, this formula is useful in assessing how $B$ depends on the various model parameters, in particular via numerical experimentation.

A similar analysis can be performed if the lethal boundary condition in $(\ref{unscaled})$ is replaced by a Neumann or zero-flux boundary condition at $X=-L/2$ and $+L/2$. The analysis is simpler, and leads to the following conclusions:
If $Q<\mu$, then the optimal harvesting rate is to fish at the maximal 
rate ${\bar H}$ everywhere (i.e. no MPA), whereas if $Q>\mu$, the optimal rate is $0$ everywhere 
(i.e. the entire coast line is a single MPA). In particular, there is no minimal coast length  warranting the implementation of an MPA like there is in the case of a lethal boundary condition. This shows that the choice of the 
boundary condition affects the optimal harvesting strategy in a major way. Current research is aimed at developing more 
realistic models of the fish dynamics such as including advection to model ocean currents, replacing the reaction terms that 
describes the natural fish dynamics, and investigating the problem in spatial dimensions higher than one.

\noindent  {\bf Acknowledgments}\\
I would like to thank VLAC (Vlaams Academisch Centrum) for hosting me as a research fellow and supporting me with a research grant,  during a sabbatical leave from the University of Florida. I am also grateful for support received from the University of Florida through a 
Faculty Enhancement Opportunity. Finally, I would like to thank the Universit{\'e} Catholique de Louvain for hosting me as a visiting professor in the 2011-2012 academic year.
\appendix
\renewcommand\thesection{Appendix \Alph{section}}

\section{Proof of Theorem $\ref{time-exp}$}
Consider the solution $\lambda(x)$ of system $(\ref{shoot})$. Since the initial condition belongs to the positive $\lambda_1$-axis, the 
vector field corresponding to $\lambda(x)$ is $F_1$, as long as the solution remains in the region ${\cal S}_a$. Let us first calculate the orbits of the system ${\dot \lambda}=F_1(\lambda)$ 
through points with coordinates $(\lambda_0,0)$. These are given by the solutions of the following ODE:
$$
\frac{d \lambda_2}{d \lambda_1}=\frac{\lambda_1}{a_1\lambda_2+1/l},
$$
and these are
$$
\lambda_1^2-a_1\lambda_2^2-\frac{2b_1}{l}\lambda_2=\lambda_0^2
$$
These orbits intersect the switching line where $\lambda_2=-1/l$,  if and only if
$$
\lambda_0^*\equiv \frac{\sqrt{2b_1-a_1}}{l}\leq \lambda_0
$$
Although solutions $\lambda(x)$ corresponding to initial conditions satisfying $0<\lambda_0<\lambda_0^*$ 
do not hit the switching line, it will be seen shortly that they do reach the $\lambda_2$-axis. 
To show this, solve the system ${\dot \lambda}=F_1(\lambda)$ with initial condition 
$(\lambda_1(0),\lambda_2(0))=(\lambda_0,0)$:
\begin{equation}\label{sol1}
\begin{pmatrix}
\lambda_1(x)\\
\lambda_2(x)
\end{pmatrix}=
\begin{pmatrix}
-\frac{b_1}{\sqrt{a_1}l}\sinh(\sqrt{a_1}x)+\lambda_0\cosh(\sqrt{a_1}x)\\
\frac{b_1}{a_1l}\left(\cosh(\sqrt{a_1}x)-1 \right)-\frac{\lambda_0}{\sqrt{a_1}}\sinh(\sqrt{a_1}x)
\end{pmatrix}
\end{equation}
Then $T_0(\lambda_0)$ can be calculated by solving $\lambda_1(T_0(\lambda_0))=0$ for $T_0(\lambda_0)$:
$$
T_0(\lambda_0)=\frac{1}{\sqrt{a_1}}\arctanh \left(\frac{\lambda_0}{i_1} \right)\textrm{ if } 0<\lambda_0<\lambda_0^*
$$
From now on focus on solutions $\lambda(x)$ corresponding to $\lambda_0\geq \lambda_0^*$. Note that these solutions 
cannot reach the $\lambda_2$-axis prior to 
crossing the switching line, since the region in the fourth quadrant that is 
to the right of the orbit of the solution corresponding to $\lambda_0=\lambda_0^*$ (which hits the switching line exactly when it reaches the $\lambda_2$-axis), and below the $\lambda_1$-axis, is forward invariant. Let us first determine in terms of the parameter $\lambda_0$, where these solutions hit the switching line. Call 
the $\lambda_1$-coordinate of this point ${\tilde \lambda_0}(\lambda_0)$. It follows from the expression of the orbits given above that
\begin{equation}\label{lambda-tilde}
{\tilde \lambda_0}(\lambda_0)=\sqrt{\lambda_0^2-(\lambda_0^*)^2}\textrm{ if }\lambda_0^*\leq \lambda_0
\end{equation}
Next, determine the smallest value of $x$ for which this happens, and call it $x_s(\lambda_0)$. It is determined by
setting $\lambda_2(x_s(\lambda_0))=-1/l$ in $(\ref{sol1})$:
$$
\frac{b_1/a_1-1}{b_1/a_1}=\cosh(\sqrt{a_1}x_s(\lambda_0))-\frac{\lambda_0}{i_1}\sinh(\sqrt{a_1}x_s(\lambda_0))
$$
Using basic properties of hyperbolic trig functions, 
solve this for $x_s(\lambda_0)$, and distinguish three cases, depending on how $\lambda_0$ compares to $i_1$:
\begin{equation}\label{switch-time}
x_s(\lambda_0)=\begin{cases}
\frac{1}{\sqrt{a_1}}\left(\arctanh\left(\frac{\lambda_0}{i_1}\right)-\arccosh \left( 
\frac{b_1/a_1-1}{b_1/a_1}\frac{1}{\sqrt{1-\left(\frac{\lambda_0}{i_1}\right)^2}}\right) \right)
\textrm{ if } \lambda_0^*\leq \lambda_0<i_1\\
-\frac{1}{\sqrt{a_1}}\ln \left( \frac{b_1/a_1-1}{b_1/a_1}\right) \textrm{ if }\lambda_0=i_1\\
\frac{1}{\sqrt{a_1}}\left(\arccoth\left(\frac{\lambda_0}{i_1}\right)-\arcsinh \left( 
\frac{b_1/a_1-1}{b_1/a_1}\frac{1}{\sqrt{\left(\frac{\lambda_0}{i_1}\right)^2-1}}\right) \right)
\textrm{ if } i_1<\lambda_0\\
\end{cases}
\end{equation}
Next find the orbits of the system ${\dot \lambda}=F_2(\lambda)$ through points with coordinates $({\tilde \lambda_0},-1/l)$ where 
${\tilde \lambda_0}>0$:
$$
\lambda_1^2-(\sqrt{a_2}\lambda_2-i_2 )^2+i_{s_2}^2={\tilde \lambda_0}^2
$$
These orbits intersect the $\lambda_2$-axis if and only if
$$
{\tilde \lambda_0}\leq i_{s_2},
$$
or equivalently, in terms of the parameter $\lambda_0$ via $(\ref{lambda-tilde})$:
$$
\lambda_0\leq \sqrt{i_{s_2}^2+(\lambda_0^*)^2}\equiv \lambda_0^{**}
$$
Since the orbits don't intersect the $\lambda_2$-axis if $\lambda_0^{**}<\lambda_0$, this implies also that
$$
T_0(\lambda_0)=+\infty \textrm{ if } \lambda_0^{**}<\lambda_0
$$
To conclude, find the value of $T_0(\lambda_0)$ for solutions $\lambda(x)$ of system $(\ref{shoot})$ that are such that 
$\lambda_0^*<\lambda_0<\lambda_0^{**}$. These are the solutions that cross the switching line once before reaching the $\lambda_2$-axis. 
The previous calculations show that they intersect the switching line at a point with coordinates $({\tilde \lambda_0},-1/l)$, where 
$0<{\tilde \lambda_0}<i_{s_2}$, and ${\tilde \lambda_0}$ can be expressed in terms of $\lambda_0$ via $(\ref{lambda-tilde})$.
Therefore, solve the system ${\dot \lambda}=F_2(\lambda)$ with initial condition 
$(\lambda_1(0),\lambda_2(0))=({\tilde \lambda_0},-1/l)$:
\begin{equation}\label{sol2}
\begin{pmatrix}
\lambda_1(x)\\
\lambda_2(x)
\end{pmatrix}=
\begin{pmatrix}
-\frac{\sqrt{a_2}}{l}\left(\frac{b_2}{a_2}-1 \right)\sinh(\sqrt{a_2}x)+{\tilde \lambda_0}\cosh(\sqrt{a_2}x)\\
\frac{1}{l}\left(\frac{b_2}{a_2}-1 \right)\cosh(\sqrt{a_2}x)
-\frac{{\tilde \lambda_0}}{\sqrt{a_2}}\sinh (\sqrt{a_2}x)
-\frac{b_2}{a_2l}
\end{pmatrix}
\end{equation}
and determine the  value of $x$ for which this solution reaches the $\lambda_2$-axis. Denote this $x$-value by $x_0({\tilde \lambda_0})$,  
and it can be calculated by setting $\lambda_1(x_0({\tilde \lambda_0}))=0$ in $(\ref{sol2})$ and solving for $x_0({\tilde \lambda_0})$:
$$
x_0({\tilde \lambda_0})=\frac{1}{\sqrt{a_2}}\arctanh \left(\frac{{\tilde \lambda_0}}{i_{s_2}} \right) \textrm{ for } 0<{\tilde \lambda_0}\leq i_{s_2}
$$
Notice that for 
$\lambda_0^*<\lambda_0<\lambda_0^{**}$, there holds that $T_0(\lambda_0)=x_s(\lambda_0)+x_0({\tilde \lambda_0})$, where 
${\tilde \lambda_0}$ can be expressed in terms of $\lambda_0$ via $(\ref{lambda-tilde})$. This finalizes the calculation of the 
function $T_0(\lambda_0)$ for all $\lambda_0>0$.

If it can be established that $T_0(\lambda_0)$ is an increasing function on the interval $(0,\lambda_0^{**})$, then the remaining statements 
in Theorem $\ref{time-exp}$ are straightforward. 
That $T_0$ is increasing if $\lambda_0 \in 
(0,\lambda_0^*)$ is obvious since $\arctanh$ is an increasing function. It will be shown that $dT_0/d\lambda_0>0$ on the intervals 
$(\lambda_0^*,i_1)$ and $(i_1,\lambda_0^{**})$, and this will conclude the proof of the Theorem.

\subsection{$\frac{dT_0}{d\lambda_0}>0$ if $\lambda_0 \in (\lambda_0^*,i_1)$}
There holds that
$$
T_0(\lambda_0)=
\frac{1}{\sqrt{a_1}}\left(\arctanh\left(\frac{\lambda_0}{i_1}\right)-\arccosh \left( 
\frac{b_1/a_1-1}{b_1/a_1}\frac{1}{\sqrt{1-\left(\frac{\lambda_0}{i_1}\right)^2}}\right) \right)+
\frac{1}{\sqrt{a_2}}\arctanh \left(\frac{\sqrt{\lambda_0^2-(\lambda_0^*)^2}}{i_{s2}} \right),
$$
and thus, since
$$
\frac{d}{dx}\left( \arctanh x\right)=\frac{1}{1-x^2}
\textrm{ and }\frac{d}{dx}\left(\arccosh x \right)=\frac{1}{\sqrt{x^2-1}},
$$
and occasionally using the definitions of $i_1$ in $(\ref{i1-en-i2})$, $i_{s,2}$ in 
$(\ref{twee-ineq})$ and $\lambda_0^*$ and $\lambda_0^{**}$ 
in $(\ref{grenzen})$ below respectively, that 
\begin{eqnarray*}
\frac{dT_0}{d\lambda_0}&=&
\frac{1}{\sqrt{a_1}}\left[\frac{1}{i_1}\frac{1}{1-(\lambda_0/i_1)^2}-\frac{b_1/a_1-1}{b_1/a_1}\frac{\lambda_0}{i_1^2}
\frac{1}{\left(1-(\lambda_0/i_1)^2 \right)^{3/2}}\frac{1}{\left(\left(\frac{b_1/a_1-1}{b_1/a_1}\right)^2\frac{1}{1-(\lambda_0/i_1)^2}-1 \right)^{1/2}}
\right]\\
&+& \frac{1}{\sqrt{a_2}i_{s_2}}\frac{\lambda_0}{(\lambda_0^2-(\lambda_0^*)^2)^{1/2}}\frac{i_{s,2}^2}{(\lambda_0^*)^2+i_{s,2}^2-\lambda_0^2}
\end{eqnarray*}
\newpage
\begin{eqnarray*}
\frac{dT_0}{d\lambda_0}
&=&
\frac{1}{\sqrt{a_1}}\left[
\frac{i_1}{i_1^2-\lambda_0^2}-\frac{\lambda_0}{(i_1^2-\lambda_0^2)\left(1-\left(\frac{b_1/a_1}{b_1/a_1-1} \right)^2
(1-(\lambda_0/i_1)^2) \right)^{1/2}}
\right]\\
&+& \frac{i_{s,2}}{\sqrt{a_2}}\frac{\lambda_0}{(\lambda_0^2-(\lambda_0^*)^2)^{1/2}}\frac{1}{(\lambda_0^{**})^2-\lambda_0^2}\\
&=&\frac{b_1}{a_1l}\frac{1}{i_1^2-\lambda_0^2}\left(1-\frac{\lambda_0/i_1}{\left(1-\left(\frac{b_1/a_1}{b_1/a_1-1} \right)^2
(1-(\lambda_0/i_1)^2) \right)^{1/2}} \right)\\
&+&(\frac{b_2}{a_2}-1)\frac{1}{l}\frac{\lambda_0}{\left( \lambda_0^2-(\lambda_0^*)^2\right)^{1/2}}
\frac{1}{(\lambda_0^{**})^2-\lambda_0^2}\\
&=&\frac{b_1}{a_1l}\frac{1}{i_1^2-\lambda_0^2}\left(1-\frac{b_1/a_1-1}{b_1/a_1}\frac{\lambda_0}{\left(\lambda_0-(\lambda_0^*)^2 \right)^{1/2}}\right)+
(\frac{b_2}{a_2}-1)\frac{1}{l}\frac{\lambda_0}{\left( \lambda_0^2-(\lambda_0^*)^2\right)^{1/2}}
\frac{1}{(\lambda_0^{**})^2-\lambda_0^2}\\
&=&\frac{1}{l}\frac{1}{\left(\lambda_0^2-(\lambda_0^*)^2 \right)^{1/2}}
\left(\frac{(b_1/a_1)\left(\lambda_0^2-(\lambda_0^*)^2\right)^{1/2}-((b_1/a_1)-1)\lambda_0}{i_1^2-\lambda_0^2}+
\frac{(b_2/a_2-1)\lambda_0}{(\lambda_0^{**})^2-\lambda_0^2}\right)
\end{eqnarray*}
The first factor in this product is positive, and thus the sign of $dT_0/d \lambda_0$ is determined by the sign of
\begin{eqnarray*}
\frac{1}{i_1^2-\lambda_0^2}\frac{1}{(\lambda_0^{**})^2-\lambda_0^2}
\left[\left((\lambda_0^{**})^2-\lambda_0^2\right)\left((b_1/a_1)\left(\lambda_0^2-(\lambda_0^*)^2\right)^{1/2}-((b_1/a_1)-1)\lambda_0\right)
+
(b_2/a_2-1)\lambda_0(i_1^2-\lambda_0^2)
\right]
\end{eqnarray*}
The first factor is positive, and so is the second, since
\begin{eqnarray*}
(\lambda_0^{**})^2-\lambda_0^2&=&(\lambda_0^*)^2+i_{s,2}^2-\lambda_0^2\\
&>&(\lambda_0^*)^2+i_{s,1}^2-\lambda_0^2\\
&=&i_1^2-\lambda_0^2>0,
\end{eqnarray*}
where the inequality $(\ref{twee-ineq})$ was used to establish the first inequality, and the definitions of 
$\lambda_0^*$ and $\lambda_0^{**}$ in $(\ref{grenzen})$, $i_{s,1}$ in $(\ref{twee-ineq})$, and  $i_1$ in $(\ref{i1-en-i2})$ in the equalities. 
Thus, the sign of $dT_0/d\lambda_0$ is determined by the sign of the third, square-bracketed factor in the product above. 
Multiply and divide it by 
$(b_1/a_1)\left(\lambda_0^2-(\lambda_0^*)^2\right)^{1/2}+((b_1/a_1)-1)\lambda_0$, and obtain -after simplifications 
using the definition of $\lambda_0^*$ in $(\ref{grenzen})$ and $i_1$ in $(\ref{i1-en-i2})$- that it equals:
\begin{equation}\label{bijna}
\frac{i_1^2-\lambda_0^2}{(b_1/a_1)\left(\lambda_0^2-(\lambda_0^*)^2\right)^{1/2}+((b_1/a_1)-1)\lambda_0}f(\lambda_0),
\end{equation}
where
\begin{equation}\label{fion}
f(\lambda_0):=
-\frac{2b_1-a_1}{a_1}\left[(\lambda_0^{**})^2-\lambda_0^2 \right]+
(b_2/a_2-1)\lambda_0\left[ (b_1/a_1)\left( \lambda_0^2-(\lambda_0^*)^2\right)^{1/2}+(b_1/a_1-1)\lambda_0\right]
\end{equation}
The fraction in $(\ref{bijna})$ is positive since $\lambda_0<i_1$ and $b_1/a_1-1>0$ (recall $(\ref{a-en-b})$). Notice also 
that since $b_i-a_i>0$ for $i=1,2$ (again by $(\ref{a-en-b})$), $f(\lambda_0)$ is an increasing function for all $\lambda_0>0$. 
Thus since $\lambda_0^*<\lambda_0$, there follows that
$$
f(\lambda_0^*)<f(\lambda_0).
$$
Expressing $\lambda_0^*$ and $\lambda_0^{**}$ in terms of the parameters $a_i$, $b_i$ and $l$, 
using $(\ref{grenzen})$ and $(\ref{twee-ineq})$, calculate
$$
f(\lambda_0^*)=\frac{2b_1-a_1}{l^2}\left(\frac{b_2}{a_2}-1 \right)\left[\frac{b_1-a_1-b_2+a_2}{a_1} \right]\equiv 0
$$
where the definitions of the $a_i$ and $b_i$, $i=1,2$, in $(\ref{a-en-b})$ were used to obtain the last equality.
Thus,
$$
0<f(\lambda_0)\textrm{ for } \lambda_0^*<\lambda_0,
$$
and therefore $dT_0/d \lambda_0>0$ whenever $\lambda_0\in (\lambda_0^*,i_1)$.

\subsection{$\frac{dT_0}{d\lambda_0}>0$ if $\lambda_0 \in (i_1,\lambda_0^{**})$}

There holds that
$$
\frac{1}{\sqrt{a_1}}\left(\arccoth\left(\frac{\lambda_0}{i_1}\right)-\arcsinh \left( 
\frac{b_1/a_1-1}{b_1/a_1}\frac{1}{\sqrt{\left(\frac{\lambda_0}{i_1}\right)^2-1}}\right) \right)+
\frac{1}{\sqrt{a_2}}\arctanh \left(\frac{\sqrt{\lambda_0^2-(\lambda_0^*)^2}}{i_{s2}} \right),
$$
and since 
$$
\frac{d}{dx}\left( \arccoth x\right)=\frac{1}{1-x^2}
\textrm{ and }\frac{d}{dx}\left(\arcsinh x \right)=\frac{1}{\sqrt{1+x^2}},
$$
similar calculations as in the previous subsection show that
\begin{eqnarray*}
\frac{dT_0}{d\lambda_0}&=&\frac{1}{\sqrt{a_1}}\left[
\frac{-i_1}{\lambda_0^2-i_1^2}+\frac{\lambda_0}{(\lambda_0^2-i_1^2)\left(1+\left(\frac{b_1/a_1}{b_1/a_1-1} \right)^2
((\lambda_0/i_1)^2-1) \right)^{1/2}}
\right]\\
&+& \frac{i_{s,2}}{\sqrt{a_2}}\frac{\lambda_0}{(\lambda_0^2-(\lambda_0^*)^2)^{1/2}}\frac{1}{(\lambda_0^{**})^2-\lambda_0^2}\\
&=&\frac{b_1}{a_1l}\frac{1}{\lambda_0^2-i_1^2}\left(-1+\frac{\lambda_0/i_1}{\left(1+\left(\frac{b_1/a_1}{b_1/a_1-1} \right)^2
((\lambda_0/i_1)^2-1) \right)^{1/2}} \right)\\
&+&(\frac{b_2}{a_2}-1)\frac{1}{l}\frac{\lambda_0}{\left( \lambda_0^2-(\lambda_0^*)^2\right)^{1/2}}
\frac{1}{(\lambda_0^{**})^2-\lambda_0^2}
\end{eqnarray*}
\newpage
\begin{eqnarray*}
\frac{dT_0}{d\lambda_0}
&=&\frac{b_1}{a_1l}\frac{1}{\lambda_0^2-i_1^2}\left(-1+\frac{b_1/a_1-1}{b_1/a_1}\frac{\lambda_0}{\left(\lambda_0-(\lambda_0^*)^2 \right)^{1/2}}\right)+
(\frac{b_2}{a_2}-1)\frac{1}{l}\frac{\lambda_0}{\left( \lambda_0^2-(\lambda_0^*)^2\right)^{1/2}}
\frac{1}{(\lambda_0^{**})^2-\lambda_0^2}\\
&=&\frac{1}{l}\frac{1}{\left(\lambda_0^2-(\lambda_0^*)^2 \right)^{1/2}}
\left(\frac{-(b_1/a_1)\left(\lambda_0^2-(\lambda_0^*)^2\right)^{1/2}+((b_1/a_1)-1)\lambda_0}{\lambda_0^2-i_1^2}+
\frac{(b_2/a_2-1)\lambda_0}{(\lambda_0^{**})^2-\lambda_0^2}\right)
\end{eqnarray*}
Since the first factor in this product is positive, the sign of $dT_0/d\lambda_0$ is determined by the sign of
$$
\frac{1}{\lambda_0^2-i_1^2}\frac{1}{(\lambda_0^{**})^2-\lambda_0^2}
\left[\left((\lambda_0^{**})^2-\lambda_0^2\right)\left(-(b_1/a_1)\left(\lambda_0^2-(\lambda_0^*)^2\right)^{1/2}+((b_1/a_1)-1)\lambda_0\right)
+
(b_2/a_2-1)\lambda_0(\lambda_0^2-i_1^2)
\right]
$$
The first factor and second factors are positive, implying that the sign of the third factor 
determines the sign of $dT_0/d \lambda_0$. Multiplying and dividing this factor by 
$(b_1/a_1)\left(\lambda_0^2-(\lambda_0^*)^2\right)^{1/2}+((b_1/a_1)-1)\lambda_0$, and obtain -after simplifications 
using the definition of $\lambda_0^*$ in $(\ref{grenzen})$ and $i_1$ in $(\ref{i1-en-i2})$- that it equals:
\begin{equation}\label{bijna2}
\frac{\lambda_0^2-i_1^2}{(b_1/a_1)\left(\lambda_0^2-(\lambda_0^*)^2\right)^{1/2}+((b_1/a_1)-1)\lambda_0}f(\lambda_0),
\end{equation}
where $f(\lambda_0)$ was defined in $(\ref{fion})$ in the previous subsection. The fraction in $(\ref{bijna2})$ is positive since 
$i_1<\lambda_0$ and $b_1/a_1-1>0$ by $(\ref{a-en-b})$. Moreover, the function $f$ is increasing for all $\lambda_0>0$, and 
it was shown in the previous subsection that $f(\lambda_0^*)=0$. But here $i_1<\lambda_0$, and since $\lambda_0^*<i_1$, it follows 
that $f(\lambda_0)>0$ for all $\lambda_0\in (i_1,\lambda_0^{**})$, and therefore $dT_0/d\lambda_0>0$ for $\lambda_0\in(i_1,\lambda_0^{**})$.

\section{Proof of Theorem $\ref{location}$}
Let ${\bar \lambda_0}$ be the unique value of $\lambda_0$ satisfying $(\ref{lambda-bar})$. From the proof of Theorem $\ref{time-exp}$ follows that ${\bar \lambda_0}\in (\lambda_0^*,\lambda_0^{**})$. This proof also reveals that 
$T_s({\bar \lambda_0})=x_s({\bar \lambda_0})$, where $x_s(\lambda_0)$ is defined in $(\ref{switch-time})$.

\end{document}